\documentclass[12pt,reqno]{amsart}

\usepackage{color}
\usepackage{amsmath, amsthm, amssymb}
\usepackage{amsfonts}
\usepackage[ansinew]{inputenc}
\usepackage[dvips]{epsfig}
\usepackage{graphicx}
\usepackage[english]{babel}
\usepackage{hyperref}
\theoremstyle{plain}
\newtheorem{thm}{Theorem}[section]
\newtheorem{cor}[thm]{Corollary}
\newtheorem{lem}[thm]{Lemma}
\newtheorem{prop}[thm]{Proposition}

\theoremstyle{definition}
\newtheorem{defi}[thm]{Definition}

\theoremstyle{remark}
\newtheorem{rem}[thm]{Remark}

\numberwithin{equation}{section}

\newcommand{\average}{{\mathchoice {\kern1ex\vcenter{\hrule height.4pt
width 6pt depth0pt} \kern-9.7pt} {\kern1ex\vcenter{\hrule
height.4pt width 4.3pt depth0pt} \kern-7pt} {} {} }}

\def\R{\mathbb{R}}

\begin{document}

\title{Regularity theory for general stable operators}

\author{Xavier Ros-Oton}
\address{The University of Texas at Austin, Department of Mathematics, 2515 Speedway, Austin, TX 78751, USA}
\email{ros.oton@math.utexas.edu}

\author{Joaquim Serra}
\address{Arcvi and Universitat Polit\`ecnica de Catalunya, Departament Matem\`atica Aplicada I, Diagonal 647, 08028 Barcelona, Spain}
\email{joaquim@arcvi.io}

\thanks{The authors were supported by grants MTM2011-27739-C04-01 (Spain), and 2009SGR345 (Catalunya)}

\maketitle

\begin{abstract}
We establish sharp regularity estimates for solutions to $Lu=f$ in $\Omega\subset\mathbb R^n$, being $L$ the generator of any stable and symmetric L\'evy process.
Such nonlocal operators $L$ depend on a finite measure on $S^{n-1}$, called the spectral measure.

First, we study the interior regularity of solutions to $Lu=f$ in $B_1$.
We prove that if $f$ is $C^\alpha$ then $u$ belong to $C^{\alpha+2s}$ whenever $\alpha+2s$ is not an integer.
In case $f\in L^\infty$, we show that the solution $u$ is $C^{2s}$ when $s\neq1/2$, and $C^{2s-\epsilon}$ for all $\epsilon>0$ when $s=1/2$.

Then, we study the boundary regularity of solutions to $Lu=f$ in $\Omega$, $u=0$ in $\mathbb R^n\setminus\Omega$, in $C^{1,1}$ domains $\Omega$.
We show that solutions $u$ satisfy $u/d^s\in C^{s-\epsilon}(\overline\Omega)$ for all $\epsilon>0$, where $d$ is the distance to $\partial\Omega$.

Finally, we show that our results are sharp by constructing two counterexamples.
\end{abstract}

\vspace{4mm}

\section{Introduction and results}

The regularity of solutions to integro-differential equations has attracted much interest in the last years, both in the Probability and in the PDE community.
This type of equations arise naturally in the study of L\'evy processes, which appear in many different situations, from Physics to Biology or Finance.

A very important class of L\'evy processes are the $\alpha$-stable processes, with $\alpha\in(0,2)$; see \cite{Bertoin} and \cite{ST}.
These are processes satisfying self-similarity properties.
More precisely, $X_t$ is said to be $\alpha$-stable if
\[  X_1 \stackrel{d}{=} \frac{1}{t^{1/\alpha}}X_t\qquad \textrm{for all}\ t > 0.\]
These processes are the equivalent to Gaussian random processes when dealing with infinite variance random variables.
Indeed, the Generalized Central Limit Theorem states that, under certain assumptions, the distribution of the sum of infinite variance random variables converges to a stable distribution (see for example \cite{ST} for a precise statement of this result).

Stable processes can be used to model real-world phenomena \cite{ST,JW}, and in particular they are commonly used in Mathematical Finance; see for example \cite{Nature,Finance,Nolan2,Nolan3,PS,BNR} and references therein.

The infinitesimal generator of any symmetric stable L\'evy process is of the form
\begin{equation}\label{operator-L}
Lu(x)=\int_{S^{n-1}}\int_{-\infty}^{+\infty}\bigl(u(x+\theta r)+u(x-\theta r)-2u(x)\bigr)\frac{dr}{|r|^{1+2s}}\,d\mu(\theta),
\end{equation}
where $\mu$ is any nonnegative and finite measure on the unit sphere, called the \emph{spectral measure}, and $s\in(0,1)$.

The aim of this paper is to establish new and sharp interior and boundary regularity results for general symmetric stable operators \eqref{operator-L}.

Remarkably, the only ellipticity assumptions in all our results will be
\begin{equation}\label{ellipt-const}
0<\lambda\leq \inf_{\nu\in S^{n-1}}\int_{S^{n-1}}|\nu\cdot\theta|^{2s}d\mu(\theta),\qquad  \int_{S^{n-1}}d\mu\leq \Lambda<\infty.
\end{equation}
Notice that these hypotheses are satisfied for \emph{any} stable operator whose spectral measure $\mu$ is $n$-dimensional, i.e., such that there is no hyperplane $V$ of $\R^n$ such that $\mu$ is supported on~$V$.
Notice also that in case that the spectral measure $\mu$ is supported on an hyperplane $V$, then no regularity result holds.

When the spectral measure is absolutely continuous, $d\mu(\theta)=a(\theta)d\theta$, then these operators can be written as
\begin{equation}\label{operator-L2}
Lu(x)=\int_{\R^n}\bigl(u(x+y)+u(x-y)-2u(x)\bigr)\frac{a(y/|y|)}{|y|^{n+2s}}\,dy,
\end{equation}
where $a\in L^1(S^{n-1})$ is a nonnegative function.

The most simple example of stable L\'evy process $X_t$ in $\R^n$ is the one corresponding to $d\mu(\theta)=c\,d\theta$, with $c>0$.
In this case, the operator $L$ is a multiple of the fractional Laplacian $-(-\Delta)^s$.
Another simple example is given by $X_t=(X^1_t,...,X^n_t)$, being $X^i_t$ independent symmetric stable processes in dimension 1.
In this case, the infinitesimal generator of $X_t$ is
\begin{equation}\label{op-deltas}
-Lu=(-\partial_{x_1x_1})^su+\cdots+(-\partial_{x_nx_n})^su,
\end{equation}
and its spectral measure consist on $2n$ delta functions.
For example, when $n=2$ we have $\mu=\delta_{(1,0)}+\delta_{(0,1)}+\delta_{(-1,0)}+\delta_{(0,-1)}$ (up to a multiplicative constant).

The regularity of solutions to $Lu=f$ (or $Lu=0$) for operators $L$ like \eqref{operator-L2}, \eqref{operator-L}, or related ones, has been widely investigated; see the works by Bass, Kassmann, Schwab, Silvestre, Sztonyk, and Bogdan, among others \cite{BassChen,KS,KRS,Bass,Sz,KM,Bass2,Silvestre,BCF,Bogdan1,Bogdan2,CS,Landkov}.
A typical assumption in some of these results is that
\begin{equation}\label{absol}
0<c\leq a(\theta)\leq C\qquad \textrm{in}\ S^{n-1}.
\end{equation}
Still, the results in \cite{KRS}, \cite{BCF}, \cite{KS}, and \cite{KM} do not require this assumption, and they apply to all operators of the form \eqref{operator-L2} satisfying
\begin{equation}\label{cone}
a(\theta)\geq c>0\quad \textrm{in a subset}\ \Sigma\subset S^{n-1} \ \textrm{of positive measure};
\end{equation}
see also \cite{SS}.
Furthermore, the results of \cite{KS} and \cite{BassChen} do not assume the spectral measure to be absolutely continuous, and apply also to the operator \eqref{op-deltas} (and also to $x$-dependent operators of the type \eqref{op-deltas}).

An important difficulty when studying the regularity for operators \eqref{operator-L} is that no Harnack inequality holds in general; see \cite{KS,BassChen} and also \cite{Bogdan1,Bogdan2}.
Also, the Fourier symbols of these operators are in general only H\"older continuous, so that the usual Fourier multiplier theorems \cite{Stein,GrubbActa,Lizorkin} can not be used to show our results.

Probably because of these difficulties, for general operators \eqref{operator-L} even the H\"older regularity of solutions was not known.
Moreover, the sharp H\"older exponent in the regularity of such solutions is only known for the case in which $\mu$ is absolutely continuous and \eqref{absol} holds.

Here we establish sharp regularity results in H\"older spaces for \emph{all} stable operators~\eqref{operator-L}.
Notice that, as explained above, for general operators \eqref{operator-L} even the H\"older continuity of solutions is new.

Our first result reads as follows.

\begin{thm}\label{thm-interior-ball}
Let $s\in(0,1)$, and let $L$ be any operator of the form \eqref{operator-L}-\eqref{ellipt-const}.
Let $u$ be any bounded weak solution to
\begin{equation}\label{eq-ball}
L u =f\quad {in}\ B_1
\end{equation}
Then,
\begin{itemize}
\item[(a)] If $f\in L^\infty(B_1)$ and $u\in L^\infty(\R^n)$,
\[\|u\|_{C^{2s}(B_{1/2})}\leq C\left(\|u\|_{L^\infty(\R^n)}+\|f\|_{L^\infty(B_1)}\right)\quad \textrm{if}\ s\neq\frac12,\]
and
\[\|u\|_{C^{2s-\epsilon}(B_{1/2})}\leq C\left(\|u\|_{L^\infty(\R^n)}+\|f\|_{L^\infty(B_1)}\right)\quad \textrm{if}\ s=\frac12,\]
for all $\epsilon>0$.
\item[(b)] If $f\in C^\alpha(B_1)$ and $u\in C^\alpha(\R^n)$ for some $\alpha>0$, then
\begin{equation}\label{estimate-thm-1-b}
\|u\|_{C^{\alpha+2s}(B_{1/2})}\leq C\left(\|u\|_{C^\alpha(\R^n)}+\|f\|_{C^\alpha(B_1)}\right)
\end{equation}
whenever $\alpha+2s$ is not an integer.
\end{itemize}
The constant $C$ depends only on $n$, $s$, and the ellipticity constants \eqref{ellipt-const}.
\end{thm}

Notice that when $s\neq\frac12$ we obtain a $C^{2s}$ estimate in part (a), and not only a $C^{2s-\epsilon}$ one.
%Even for the fractional Laplacian $(-\Delta)^s$, this result seems to be new, since to our knowledge it was only known a $C^{2s-\epsilon}$ estimate when the right hand side $f$ is $L^\infty$ (see for example \cite[Proposition 2.9]{S-obst}).

Note also that in part (b) it is required that $u\in C^\alpha(\R^n)$ in order to have a $C^{\alpha+2s}$ estimate for $u$ in $B_{1/2}$.
When the spectral measure $\mu$ is not regular, the estimate is not true anymore if $u$ is not $C^\alpha$ in all of $\R^n$: we can construct a solution to $Lu=0$ in $B_1$, which satisfies $u\in C^{\alpha-\epsilon}(\R^n)$ but $u\notin C^{\alpha+2s}(B_{1/2})$; see Proposition \ref{contraexemple-interior}.

When the spectral measure is $C^\alpha(S^{n-1})$, then it is easy to see that one can replace the $C^\alpha(\R^n)$ norm of $u$ in \eqref{estimate-thm-1-b} by the $L^\infty(\R^n)$ norm; see Corollary \ref{cor-interior-ball}.
Also, when the equation is posed in the whole $\R^n$ then there is no such problem, and one has the estimate $\|u\|_{C^{\alpha+2s}(\R^n)}\leq C(\|u\|_{L^{\infty}(\R^n)}+\|f\|_{C^{\alpha}(\R^n)})$ ---which follows easily from \eqref{estimate-thm-1-b}.

Concerning the boundary regularity of solutions, our main result reads as follows.

\begin{thm}\label{thm-bdry-reg}
Let $s\in(0,1)$, $L$ be any operator of the form \eqref{operator-L}-\eqref{ellipt-const}, and $\Omega$ be any bounded $C^{1,1}$ domain.
Let $f\in L^\infty(\Omega)$, and $u$ be a weak solution of
\begin{equation}\label{eq-I-flat}
\left\{ \begin{array}{rcll}
L u &=&f&\textrm{in }\Omega \\
u&=&0&\textrm{in }\R^n\setminus\Omega.
\end{array}\right.
\end{equation}
Let $d$ be the distance to $\partial\Omega$.
Then, $u\in C^s(\R^n)$, and
\[\|u/d^s\|_{C^{s-\epsilon}(\overline\Omega)}\leq C\|f\|_{L^\infty(\Omega)}\]
for all $\epsilon>0$.
The constant $C$ depends only on $n$, $s$, $\Omega$, and the ellipticity constants \eqref{ellipt-const}.
\end{thm}

For general stable operators \eqref{operator-L}, we expect this result to be optimal.
Indeed, we can construct a $C^\infty$ domain $\Omega$ for which $L(d^s)$ does not belong to $L^\infty(\Omega)$; see Proposition \ref{prop-appendix}.
Thus, even in $C^\infty$ domains and with $f\in C^\infty$, we do not expect solutions $u$ to satisfy $u/d^s\in C^s(\overline\Omega)$.

The estimate of Theorem \ref{thm-bdry-reg} was only known in case that the spectral measure $\mu$ is absolutely continuous and satisfies quite strong regularity assumptions.
Indeed, when \eqref{absol} holds, $a\in C^{1,\alpha}(S^{n-1})$, and $\Omega$ is $C^{2,\alpha}$, then the result is a particular case from our estimates in \cite{RS-K} for fully nonlinear equations.
Also, when $\Omega$ is $C^\infty$ and $a\in C^\infty(S^{n-1})$ then Theorem \ref{thm-bdry-reg} follows from the results of Grubb \cite{Grubb,Grubb2} for pseudodifferential operators satisfying the $\mu$-transmission property.

Even for the fractional Laplacian, the proof we present here is new and completely independent with respect to the ones in \cite{RS-Dir} and \cite{Grubb,Grubb2}.
Let us explain briefly the main ideas in the proofs of our results.

To prove Theorems \ref{thm-interior-ball} and \ref{thm-bdry-reg} we use some ideas introduced in \cite{Serra,RS-K,S-convex}.
Namely, all the proofs of the present paper have a similar structure in which we first establish a Liouville-type theorem in $\R^n$ (or $\R^n_+$ in case of boundary regularity), and then we deduce by a blow up and compactness argument an estimate for solutions to $Lu=f$ in, say, $B_1$.
An important difference with respect to the proofs \cite{Serra,RS-K,S-convex} is that here we do not have any $C^\gamma$ estimate that we can iterate in order to prove a Liouville theorem, and hence the proofs of the present Liouville theorems must be completely different from the ones in \cite{Serra,RS-K,S-convex}.

For example, in case of Theorem \ref{thm-interior-ball}, to prove the Liouville-type Theorem \ref{Liouv-entire} we need to establish first a $C^\gamma$ estimate in $\R^n$ via the heat kernel of the operator, to then iterate it and deduce the Liouville theorem.
Recall that even this first $C^\gamma$ estimate is new for general operators \eqref{operator-L}.
In case of Theorem \ref{thm-bdry-reg}, we also prove the Liouville-type Theorem \ref{Liouv-half} in a different way with respect to \cite{RS-K}.
Indeed, in \cite{RS-K} we first established a $C^\gamma$ estimate for $u/d^s$ by using a method of Caffarelli, which relies mainly on the Harnack inequality, and then we deduced from this a Liouville theorem in $\R^n_+$.
However, in the present context we do not have any Harnack inequality, and we have to establish Theorem~\ref{Liouv-half} using only the interior estimates for $u$ previously proven in Theorem \ref{thm-interior-ball}.

All the regularity estimates of this paper are for translation invariant equations.
Still, the methods presented here can be used to establish similar regularity results for non translation invariant equations (with continuous dependence on $x$), and also for parabolic equations $\partial_t u+Lu=f$ in $\Omega\times(0,T)$.
We plan to do this in a future work.

The paper is organized as follows.
In Section \ref{sec2} we establish a Liouville-type theorem in the entire space, Theorem \ref{Liouv-entire}.
In Section \ref{sec3} we prove Theorem \ref{thm-interior-ball}.
Then, in Section \ref{sec5} we establish a Liouville-type theorem in the half-space, Theorem \ref{Liouv-half}, and in Section \ref{sec6} we prove Theorem \ref{thm-bdry-reg}.
Finally, in Section \ref{sec7} we prove Proposition \ref{prop-appendix}.

\section{A Liouville theorem in the entire space}
\label{sec2}

The aim of this section is to prove the following.

\begin{thm}\label{Liouv-entire}
Let $s\in(0,1)$, and let $L$ be any operator of the form \eqref{operator-L}-\eqref{ellipt-const}.
Let $u$ be any weak solution of
\[Lu=0\quad \textrm{in}\ \R^n\]
satisfying the growth condition
\[\|u\|_{L^\infty(B_R)}\leq CR^{\beta}\qquad \textrm{for all}\ R\geq1,\]
for some $\beta<2s$.

Then, $u$ is a polynomial of degree at most $\lfloor \beta \rfloor$, where $\lfloor x \rfloor$ denotes the integer part of $x$.
\end{thm}

This Liouville theorem will be used in the proof of Theorem \ref{thm-interior-ball}.
For related Liouville theorems, see \cite{FW,Fall,Chen-Liouv}.

\addtocontents{toc}{\protect\setcounter{tocdepth}{1}}  %això és perquè no surtin les subsections al tableofcontents

\subsection{Heat kernel: regularity and decay in average}

The heat kernel for the operator $L$ is defined via Fourier transform as
\begin{equation}\label{kernel-fourier}
p(t,\cdot ) = \mathcal F^{-1}\bigl( \exp (- A(\xi)t) \bigr),
\end{equation}
where $A(\xi)$ is the Fourier symbol of the operator $L$.

The symbol $A(\xi)$ of $L$ can be explicitly written in terms of $s$ and the spectral measure $\mu$.
Indeed, it is given by
\begin{equation}\label{symbol}
A(\xi)=\int_{S^{n-1}}|\xi\cdot\theta|^{2s}d\mu(\theta);
\end{equation}
see for example \cite{ST}.
Notice that $A(\xi)$ is homogeneous of order $2s$.

In order to prove Theorem \ref{Liouv-entire}, we will need to show some kind of decay for the heat kernel of $L$.

The decay of the heat kernel has been studied in \cite{Dz} and \cite{GH} in case that $d\mu(\theta)=a(\theta)d\theta$ (see also \cite{Bogdan2,W}).
It turns out that, when $a\in L^\infty(S^{n-1})$, the heat kernel $p(t,x)$ associated to the operator \eqref{operator-L2} satisfies
\begin{equation}\label{bound-heat}
p(1,x)\leq \frac{C}{1+|x|^{n+2s}}.
\end{equation}
However, for general operators \eqref{operator-L}, the heat kernel does not satisfy in general \eqref{bound-heat}.
For example, when $X_t=(X^1_t,...,X^n_t)$, being $X^i$ independent symmetric stable processes in dimension 1, $p$ satisfies
\[p(t,x)=p_1(t,x_1)\cdots p_1(t,x_n),\]
and thus it does not satisfy \eqref{bound-heat}.

We prove here that for general operators \eqref{operator-L}, even if there is no decay of the form \eqref{bound-heat}, the heat kernel $p(1,x)$ decays ``in average'' faster than $|x|^{-n-2s+\delta}$ for any $\delta>0$.
This is stated in the following result.

\begin{prop}\label{heat-kernel-decay}
Let $s\in(0,1)$, and let $L$ be any operator of the form \eqref{operator-L}-\eqref{ellipt-const}.
Let $p(t,x)$ be the heat kernel associated to $L$.
Then,
\begin{itemize}
\item[(a)] For all $\delta>0$,
\begin{equation}\label{avdecay}
\int_{\R^n}\bigl(1+|x|^{2s-\delta}\bigr)p(1,x)dx\leq C.
\end{equation}
\item[(b)] Moreover,
\[ [p(1,x)]_{C^{0,1}(\R^n)} \le C.\]
\end{itemize}
The constant $C$ depends only on $n$, $s$, $\delta$, and the ellipticity constants \eqref{ellipt-const}.
\end{prop}

\begin{proof} (a)
We first claim that the function
\[\varphi(x)=(1+|x|^2)^{s-\delta}\]
satisfies
\[|L\varphi|\le C\qquad\textrm{in all of}\ \R^n.\]
Indeed, observe that for all $\rho\ge 1$, the rescaled function $\varphi_\rho(x) =\rho^{-2s+2\delta}\varphi(\rho x)$ satisfies $\varphi_\rho(x) =(\rho^{-2}+|x|^2)^{s-\delta}$ and $|L\varphi_\rho|\le C$ in $B_2\setminus B_1$, with $C$ independent of $\rho$.
Therefore, scaling back we obtain that $|L\varphi|\le C\rho^{-2\delta}$ in $B_{2\rho} \setminus B_{\rho}$ for every $\rho\geq1$.
Hence, $L\varphi$ is bounded in all of $\R^n$, as claimed.

Now, we have
\begin{equation}\label{potrom}\begin{split}
\int_{\R^n}  \varphi(x) p(1,x)\,dx-1 &= \int_{\R^n} \varphi(x)\bigl(p(1,x)-p(0,x)\bigr)dx= \int_0^1dt \int_{\R^n} \varphi(x) p_t(t,x) dx \\
&= \int_0^1 dt \int_{\R^n} \varphi(x) Lp(t,x)dx=\int_0^1 dt \int_{\R^n} L\varphi(x)p(t,x)dx.
\end{split}\end{equation}
Thus, it follows that
\[\int_{\R^n}  \varphi(x) p(1,x)\,dx\leq 1+\int_0^1 dt \int_{\R^n} |L\varphi(x)|p(t,x)dx\leq C,\]
and (a) follows.

(Note that $C$ depends only on $n$, $s$, $\lambda$, and $\Lambda$.
Hence, in order to justify rigorously the last integration by parts in \eqref{potrom} we may assume first that $\mu(d\theta)= a(\theta)d\theta$, with $a\in L^\infty$ ---so that $p$ and all its derivatives decay---, and then by approximation the same identity holds for any spectral measure $\mu$.)

(b)
Notice that, by \eqref{symbol} and by definition of the ellipticity constants \eqref{ellipt-const}, we clearly have
\[0<\lambda|\xi|^{2s}\leq A(\xi)\leq \Lambda|\xi|^{2s}.\]
Using this, it follows immediately from the expression \eqref{kernel-fourier} that the Fourier transform of $p(1,x)$ is rapidly decreasing and, therefore, the result follows.
\end{proof}

\begin{rem}
In case that $L$ is an operator of the form \eqref{operator-L2} and $a$ belongs to the space $L\log L(S^{n-1})$, Proposition \ref{heat-kernel-decay} (a) is an immediate consequence of the results of Glowacki-Hebisch \cite{GH}.
Indeed, it was proved in \cite{GH} that, under this assumption on $a$, the heat kernel satisfies $p(1,x)\leq C|x|^{-n-2s}\omega(x/|x|)$ for some function $\omega\in L^1(S^{n-1})$.
\end{rem}

\subsection{Proof of Theorem \ref{Liouv-entire}}

Using Proposition \ref{heat-kernel-decay}, we can now give the:

\begin{proof}[Proof of Theorem \ref{Liouv-entire}]
Given $\rho\ge1$ let
\[v(x) = \rho^{-\beta}u(\rho x).\]
 Then, $v$ clearly satisfies $Lv=0$ in the whole $\R^n$. Moreover,
\begin{equation}\label{gctrlv}
 \|v\|_{L^\infty(B_R)} =  \|\rho^{-\beta}u\|_{L^\infty(B_\rho R)} \le C\rho^{-\beta}(\rho R)^{\beta}\le CR^{\beta}.
 \end{equation}
Then, formally we have
\[ v -p(1,\cdot )\ast v = \bigl[ p(t,\cdot)\ast v \bigr]_{t=0}^{t=1} = \int_0^1 \partial_t p\ast v \,dt = \int_0^1 Lp\ast v dt = \int_0^1 p\ast Lv =0\]
and thus
\begin{equation}\label{cosa-v}
 v \equiv  p(1,\cdot)\ast v.
\end{equation}
This computation is formal, since we did not checked that the integrals defining the convolutions are finite and since $Lv$ is in principle only defined in weak sense (in the sense of distributions).

To prove rigorously \eqref{cosa-v},  we have to do the previous computation in the weak formulation, as follows.
Let
\[V(x,t)= (p(t,\cdot)\ast v) (x).\]
Then, using the growth control on $v$ and Proposition \ref{heat-kernel-decay} (a), it follows that $V$ is a weak solution of $V_t=LV$ in $(0,+\infty)\times \R^n$.
Thus, for all $\eta\in C^\infty_c\bigl((0,1)\times\R^n\bigr)$ we have
\begin{equation}\label{order-integration}\begin{split}
-\int_{0}^1\int_{\R^n} V \eta_t \,dx\,dt&=  \int_{0}^1\int_{\R^n} V L \eta \,dx\, dt
\\&= \int_{0}^1 \int_{\R^n} p(t,z) \int_{\R^n} v(x-z) L \eta(x,t) \,dx\, dz\, dt = 0.
\end{split}\end{equation}
In the last identity we have used that $\int_{\R^n} v(x-z) L \eta(x,t) dx =0$ for all $x$ and $t$, which follows from the fact that $v$ is a weak solution of $Lv=0$ in the whole $\R^n$.

Let us justify in detail the change in the order of integration in \eqref{order-integration}.
First, observe that the growth control of $v$ \eqref{gctrlv} implies that  $\int_{\R^n} |v(x-z)|  \,|L \eta(x,t)| dx \le C (1+ |z|)^{\beta}$, with $C$ depending on $\eta$ and on the constant in the growth control.
Therefore,
\[\int_{0}^1 \int_{\R^n} t^{-\frac{n}{2s}}p(1, zt^{-\frac{1}{2s}}) \int_{\R^n} |v(x-z)| \,|L \eta(x,t)| dx dz dt \le C\int_{\R^n} p(1, z)(1+|z|)^{\beta}dz < \infty.\]
Hence, we can use Fubini in \eqref{order-integration} to change the order of the integrals, as desired.
Thus, \eqref{cosa-v} is proved.

Let us now show that
\begin{equation}\label{holder-v}
[v]_{C^\gamma(B_1)}\le C
\end{equation}
for some $\gamma>0$ and $C$ depending only on $n$, $\lambda$, $\Lambda$, and $\beta$.

Indeed, given $x,x'\in B_1$ with $x\neq x'$, we have
\[
\begin{split}
|v(x)-&v(x')| = \left| p(1,\cdot)\ast v(x) - p(1,\cdot)\ast v(x') \right| = \left| \int_{\R^n} \bigl(p(x-y)-p(x'-y) \bigr)v(y) dy   \right|
\\
&\le   \left| \int_{|y|\le M} \bigl(p(x-y)-p(x'-y) \bigr)v(y) dy   \right| +  2\sup_{x\in B_1} \left| \int_{|y|\ge M} p(x-y)v(y) dy   \right|.
\end{split}
\]
To bound the first term in the right hand side of the inequality, we use Proposition~\ref{heat-kernel-decay} (b) and also \eqref{gctrlv} to find
\[\left| \int_{|y|\le M} \bigl(p(x-y)-p(x'-y) \bigr)v(y) dy   \right|\leq CM^{n+\beta}|x-x'|.\]
To bound the second term, we use Proposition \ref{heat-kernel-decay} (a), with $\delta>0$ such that $2\delta=2s-\beta$.
Using also  \eqref{gctrlv}, we find that
\[\left| \int_{|y|\ge M} p(x-y)v(y) dy   \right|\leq \int_{|y|\ge M} p(x-y)(1+|x|)^{2s-\delta} \frac{|v(y)|}{(1+|x|)^{\beta+\delta}} dy\leq CM^{-\delta}.\]
Thus, we have proved
\[|v(x)-v(x')|\leq  CM^{n+\beta}|x-x'|+ CM^{-\delta}.\]
Since this can be done for any $M>0$, we may choose
\[M=|x-x'|^{-\gamma/\delta}, \qquad \textrm{with}\ \, 1-(n+\beta)\gamma/\delta = \gamma.\]
Then, we have
\[|v(x)-v(x')|\leq C|x-x'|^\gamma,\]
and $\gamma>0$.

This shows \eqref{holder-v}.
Equivalently, what we have proved can be written as
\[ [u]_{C^\gamma(B_\rho)} \le C\rho^{\beta-\gamma}\qquad\textrm{for all}\ \rho \ge 1.\]

Next we consider the incremental quotient
\[ u_h^\gamma = \frac{u(\cdot+h)-u}{|h|^\gamma} \]
which grows (by the last inequality) as $\|u_{h}^\gamma\|_{L^\infty(B_R)}\le CR^{\beta-\gamma}$.
Then we can repeat the previous argument with $v$ replaced by $u_h^{\gamma}$ and $\beta$ replaced by $\beta-\gamma$
to show that $[u_h^\gamma]_{C^\gamma(B_R)} \le C R^{\beta-2\gamma}$, and thus
\[ [u]_{C^{2\gamma}(B_R)} \le CR^{\beta-2\gamma}.\]

We keep iterating in this way until after $N$ steps we find
\[ [u]_{C^{N\gamma}(B_R)} \le CR^{\beta-N\gamma}.\]
Taking $N$ the least integer such that $\beta -N\gamma<0$ and sending $R\rightarrow +\infty$, we obtain
\[ [u]_{C^{N\gamma}(\R^n)} =0,\]
and this implies that $u$ is a polynomial of degree at most $\lfloor \beta\rfloor$.
\end{proof}

Finally, we give a consequence of Theorem \ref{Liouv-entire} that will be also needed in the proof of Theorem \ref{thm-interior-ball}.

\begin{cor}\label{Liouv-entire-2}
Let $s\in(0,1)$, $\alpha\in(0,1)$, and $L$ be any operator of the form \eqref{operator-L}-\eqref{ellipt-const}.
Let $u$ be any function satisfying, in the weak sense,
\[L[u(\cdot+h)-u(\cdot)]=0\quad \textrm{in}\ \R^n,\qquad \textrm{for all}\ h\in\R^n.\]
Assume that $u$ satisfies the growth condition
\[[u]_{C^\alpha(B_R)}\leq CR^{\beta}\qquad \textrm{for all}\ R\geq1,\]
for some $\beta<2s$.

Then, $u$ is a polynomial of degree at most $\lfloor \beta+\alpha\rfloor$.
\end{cor}

\begin{proof}
Apply Theorem \ref{Liouv-entire} to incremental quotients of order $\alpha$ of~$u$.
\end{proof}

\section{Interior regularity}
\label{sec3}

The aim of this section is to prove Theorem \ref{thm-interior-ball}.
For it, we will use a compactness argument and the Liouville theorems established in the previous section.

We start with the following.

\begin{lem}\label{lem-subseq}
Let $s\in(0,1)$, and let $\lambda$ and $\Lambda$ be fixed positive constants.
Let $\{L_k\}_{k\geq1}$ be any sequence of operators of the form \eqref{operator-L} whose spectral measures satisfy \eqref{ellipt-const}.

Then, a subsequence of $\{L_k\}$ converges weakly to an operator $L$ of the form \eqref{operator-L} whose spectral measure satisfies \eqref{ellipt-const}.
\end{lem}

\begin{proof}
Let $\{\mu_k\}_{k\geq1}$ be the spectral measures of the operators $L_k$.
Using the weak compactness of probability measures on the sphere, we find that there is a subsequence $\mu_{k_m}$ converging to a measure $\mu$ that satisfies \eqref{ellipt-const}.

Let $L$ be the operator given by \eqref{operator-L} whose spectral measure is $\mu$.
Then, we have that the subsequence $L_{k_m}$ converge weakly to $L$.
Indeed, for any test function $w\in C^\infty_c(\R^n)$  we have
\[|2w(x)-w(x+y)-w(x-y)|\leq C\bigl(|y|^2\wedge 1\bigr),\]
and thus it follows from the dominated convergence theorem that
\[L_{k_m}w\longrightarrow Lw\]
uniformly in compact sets of $\R^n$.
\end{proof}

We next establish the following result, which is the main step towards Theorem \ref{thm-interior-ball} (b).

\begin{prop}\label{claim-a}
Let $s\in(0,1)$, and let $L$ be any operator of the form \eqref{operator-L}-\eqref{ellipt-const}.
Let $\alpha\in(0,1)$ be such that $\alpha+2s$ is not an integer.
Let $\alpha'\in (0,\alpha)$ be such that $\lfloor \alpha+2s\rfloor<\alpha'+2s<\alpha+2s$ and that $\alpha<\alpha'+2s$.

Let $w$ be any $C_c^\infty(\R^n)$ satisfying $L w=f$ in $B_1$, with $f\in C^\alpha(B_1)$.
Then, we have the estimate
\begin{equation}\label{estw}
[w]_{C^{\alpha+2s}(B_{1/2})} \le C\bigl( [f]_{C^\alpha(B_1)} + \|w\|_{C^{\alpha'+2s}(\R^n)}\bigr).
\end{equation}
The constant $C$ depends only on $n$, $s$, $\alpha$, $\alpha'$, and the ellipticity constants \eqref{ellipt-const}.
\end{prop}

\begin{proof}
The proof of \eqref{estw} is by contradiction.
If the statement of the proposition is false then, for each integer $k\ge 0$, there exist  $L_k$, $w_k$, and $f_{k}$ satisfying:
\begin{itemize}
\item $ L_kw_k =f_k $   in $B_{1}$;
\vspace{3pt}
\item $L_k$ is of the form \eqref{operator-L}-\eqref{ellipt-const};
\vspace{3pt}
\item $[f_k]_{C^\alpha(B_1)} +\| w_k \|_{C^{2s+{\alpha'}}(\R^n )} \le 1$ (we may  always assume this dividing $w_k$ by  the previous quantity);
\item $\|w_k\|_{C^{{2s}+\alpha}(B_{1/2})} \ge k$.
\end{itemize}
In the rest of the proof we denote
\[\nu=\lfloor \alpha+2s\rfloor.\]

Since $\nu<\alpha'+2s<\alpha+2s$ we then have
\begin{equation}\label{2k2}
\sup_k \sup_ {z\in B_{1/2}} \sup_{r>0} \ r^{{\alpha'}-\alpha}\left[ w_k \right]_{C^{{2s}+{\alpha'}}(B_{r}(z))} = +\infty.
\end{equation}

Next, we define
\[ \theta(r) := \sup_k  \sup_ {z\in  B_{1/2}}  \sup_{r'>r}  (r')^{{\alpha'}-\alpha}\,\bigl[w_k\bigr]_{C^{{2s}+{\alpha'}} \left(B_{r'}(z)\right)}.\]
The function $\theta$ is monotone nonincreasing, and we have $\theta(r)<+\infty$ for $r>0$ since we are assuming that  $\|w_k\|_{C^{{2s}+{\alpha'}}(\R^n)}\le 1$.
In addition, by \eqref{2k2} we have  $\theta(r)\rightarrow +\infty$ as $r\downarrow0$.

Now, for every positive integer $m$, by definition of $\theta(1/m)$ there exist $r'_m\ge 1/m$, $k_m$, and $z_{m} \in  B_{1/2}$, for which
\begin{equation}\label{nondeg2}
(r'_m)^{{\alpha'}-\alpha} \bigl[w_{k_m}\bigr]_{C^{{2s}+{\alpha'}} \left(B_{r'_m}(z_m)\right)} \ge \frac{1}{2}\,\theta(1/m) \ge \frac{1}{2}\,\theta(r'_m).
\end{equation}
Here we have used that $\theta$ is nonincreasing.
Note that we will have $r'_m\downarrow0$.

Let $p_{k,z,r}(\cdot\,-z)$ be the polynomial of degree less or equal than $\nu$  in the variables $(x-z)$ which best fits $u_k$ in $B_r(z)$ by least squares. That is,
\[p_{k,z,r} := {\rm arg\,min}_{p\in \mathbb P_\nu} \int_{B_r(z)} \bigl(w_k(x)-p(x-z)\bigr)^2 \,dx ,\]
where $\mathbb P_\nu$ denotes the linear space of polynomials of degree at most $\nu$ with real coefficients.
From now on in this proof we denote
\[p_m = p_{k_m, z_m,r'_m}.\]

We consider the blow up sequence
\begin{equation}\label{eqvm}
 v_m(x) = \frac{w_{k_m}(z_{m} +r'_m x)-p_{m}(r'_m x)}{(r'_m)^{{2s}+\alpha}\theta(r'_m)}.
 \end{equation}
Note that, for all $m\ge 1$ we have
\begin{equation}\label{2}
\int_{B_1(0)} v_m(x) q(x) \,dx =0\quad \mbox{for all } q\in \mathbb P_\nu.
\end{equation}
This is the optimality condition for least squares.
Note also that \eqref{nondeg2} implies the following nondegeneracy condition for all $m\geq1$:
\begin{equation}\label{nondeg35}
[v_m]_{C^{{2s}+{\alpha'}}(B_1)}\ge 1/2.
\end{equation}

Next, we can estimate
\[\begin{split}
[v_{m}]_{C^{{2s}+{\alpha'}}(B_R)} &= \frac{1}{\theta({r'_m})(r'_m)^{\alpha-{\alpha'}} } \bigl[w_{k_m} \bigr]_{C^{{2s}+{\alpha'}}\left(B_{Rr'_m}(z_m)\right)}
\\
&= \frac{R^{\alpha-{\alpha'}}}{\theta({r'_m}) (Rr'_m )^{\alpha-{\alpha'}} }\bigl[w_{k_m}\bigr]_{C^{{2s}+{\alpha'}}\left(B_{Rr'_m}(z_m)\right)} .
\end{split}
\]
Indeed, the definition of $\theta$ and its monotonicity yield the following growth control for the $C^{{2s}+{\alpha'}}$ seminorm of $v_m$
\begin{equation}\label{growthc0}
[v_{m}]_{C^{{2s}+{\alpha'}} (B_R)} \leq CR^{\alpha-{\alpha'}}\quad \textrm{for all}\ \,R\ge 1.
\end{equation}

When $R=1$, \eqref{growthc0} implies that $\|v_m- q\|_{L^\infty(B_1)}\le C$, for some $q\in \mathbb P_\nu$.
Therefore, \eqref{2} yields
\begin{equation}\label{boundedinB1}
\|v_m\|_{L^\infty(B_1)}\le C.
\end{equation}

Now, we will see that using \eqref{growthc0}-\eqref{boundedinB1} we obtain
\begin{equation}\label{growthc1}
 [v_{m}]_{C^{\gamma} (B_R)} \leq CR^{{2s}+\alpha-\gamma}\qquad \textrm{for all}\ \gamma\in[0, {2s}+{\alpha'}]
\end{equation}
Indeed, \eqref{boundedinB1} implies that for every multiindex $l$ with $|l|\le\nu$ there is some point $x_* \in B_1$  such that
\[|D^l v_m(x_*)|\le C,\qquad x_*\in B_1.\]
The existence of such $x_*$ can be shown taking some nonnegative $\eta\in C^\infty_c(B_1)$ with unit mass and observing that
\[ \left|\int \eta (x) D^l v_m (x) \,dx\right| \le C \int |D^l \eta| v_m (x)\,dx \le C.\]
Hence, using \eqref{growthc0}, for all $l$ with $|l|=\nu$ and $x\in B_R$ we have
\[  |D^l v_m (x) | \le |D^l v_m (x^*)| + C R^{\alpha-{\alpha'}} |x-x^*|^{{2s}+{\alpha'}-\nu}  \le CR^{{2s}+\alpha-\nu}.\]
Iterating the same argument one can show the corresponding estimate for $|l|= \nu-1, \, \nu-2,$ etc.
Then, once established \eqref{growthc1} for all integer $\gamma\in [0,{2s}+{\alpha'}]$, the result for all $\gamma$ follows by interpolation.
Thus, \eqref{growthc1} is proved.

We now prove the following:

\vspace{6pt}
\noindent {\em Claim I.} The sequence $v_m$ converges in $C^{{2s}+{\alpha'}/2}_{\rm loc}(\R^n)$ to a function $v\in C^{{2s}+{\alpha'}}_{\rm loc}(\R^n)$. This function $v$ satisfies the assumptions of the Liouville-type Corollary \ref{Liouv-entire-2}.
\vspace{6pt}

The $C^{{2s}+{\alpha'}/2}$ uniform convergence on compact sets of $\R^n$ of the function $v_m$ to some $v\in C^{{2s}+{\alpha'}}(\R^n)$ follows from \eqref{growthc1} and the Arzel\`a-Ascoli theorem (and the usual diagonal sequence argument).
Moreover, passing to the limit \eqref{growthc1} with $\gamma\in (\alpha,1]$ such that $\gamma\leq \alpha'+2s$, we find
\begin{equation}\label{growth-limit-fnct-v}
[v]_{C^{\gamma}(B_R)}\leq CR^{\beta}\qquad\textrm{for all}\ R\geq1,
\end{equation}
$\beta=2s+\alpha-\gamma<2s$.
Thus, $v$ satisfies the growth assumption in Corollary \ref{Liouv-entire-2}.

On the hand, each $w_k$ satisfies a $L_kw_k =f_k$ in $B_1$.
Thus,  recalling that we have $[f_k]_{C^\alpha(B_1)}\le 1$, we find that

\begin{equation}\label{uuu}
\bigl|L_k w_k(\bar x+ \bar h)- L_k w_k(\bar x)\bigr|\leq |\bar h|^\alpha\qquad
\textrm{for all}\ \bar x\in B_{1/2}(z)\ \textrm{and}\ \bar h\in B_{1/2}.
\end{equation}

Note now that, since $\nu\le 2$,
\begin{equation}\label{difpols}
\delta^2 p(x+h,y)-\delta^2 p(x,y) = 0\quad \mbox{for all }p\in \mathbb P_\nu\mbox{  and for all }x,y,h\mbox{ in }\R^n.
\end{equation}
Here, as usual, we have denoted $\delta^2\varphi(x)=\varphi(x+y)+\varphi(x-y)-2\varphi(x)$.

Next, taking into account \eqref{difpols}, we translate \eqref{uuu} from $w_{k_m}$ to $v_{m}$.
Namely, using the definition of $v_m$ in \eqref{eqvm}, and setting $\bar h= r'_m h$,  and $\bar x=z_m+ r'_mx$ in \eqref{uuu}, we obtain
\[
\frac{1}{(r'_m)^{2s}}\left| L_{k_m} \left(  (r'_m)^{{2s}+\alpha} \theta(r'_m)\left\{ v_{m}(\,\cdot\,+ h)  -v_{m}\right\}\right)(x)\right|\leq (r'_m)^\alpha  |h|^{\alpha}
\]
whenever  $|x| \le \frac{1}{2r'_m}$, and thus
\begin{equation}\label{vv}
\left| L_{k_m} \left( v_{m}(\,\cdot\,+ h)-v_{m}\right)(x)\right|\leq \frac {1}{\theta(r'_m)} \quad \mbox{whenever } |x| \le \frac{1}{2 r'_m}.
\end{equation}

By Lemma \ref{lem-subseq}, the operators $L_{k_m}$ converge weakly (up to subsequence) to an operator $L$.
Thus, passing \eqref{vv} to the limit we find that
\[L \left(  v(\,\cdot\,+ h)-v\right)=0 \quad \mbox{in all of }\R^n.\]

Notice that to be able to pass to the limit $m \to +\infty$ on the right hand side of \eqref{vv} we are using that, by \eqref{growthc1}, the functions
$v_{k_m}(\,\cdot\,+ h)-v_{k_m}$ satisfy
\[\|v_{k_m}(\,\cdot\,+ h)-v_{k_m}\|_{C^{2s+\alpha'}(B_R)}\leq C(R),\]
and also the growth control
\[\|v_{k_m}(\,\cdot\,+ h)-v_{k_m}\|_{L^\infty(B_R)}\leq CR^{2s-\epsilon}\qquad\textrm{for all}\ R\geq1,\]
for some $\epsilon>0$ (this follows from \eqref{growthc1}).

This finishes the proof of Claim.
\vspace{5pt}

We have thus proved that the limit function $v$ satisfies the assumptions of Corollary \ref{Liouv-entire-2}, and hence we conclude that $v$ is a polynomial of degree $\nu$.
On the other hand, passing \eqref{2} to the limit we obtain that $v$ is orthogonal to every polynomial of degree $\nu$ in $B_1$, and hence it must be $v\equiv 0$.
But then passing \eqref{nondeg35} to the limit we obtain that $v$ cannot be constantly zero in $B_1$; a contradiction.
\end{proof}

We can now give the:

\begin{proof}[Proof of Theorem \ref{thm-interior-ball} (b)]
Let $\nu=\lfloor \alpha+2s\rfloor$, and let $\alpha'$ be such that $\nu <\alpha'+2s$.
Such $\alpha'$ exists because $\alpha+2s$ is not an integer (by assumption).
We will deduce the theorem from Proposition \ref{claim-a}, as follows.

First, it immediately follows from Proposition \ref{claim-a} that for any $w\in C^\infty_c(\R^n)$,
\begin{equation}\label{pas1}
[w]_{C^{\alpha+2s}(B_{1/2})} \le C\bigl( [f]_{C^\alpha(B_1)} + [w]_{C^{\alpha'+2s}(B_2)}+\|w\|_{C^\alpha(\R^n)}\bigr).
\end{equation}
To prove this, take a cutoff function $\eta\in C^\infty_c(B_2)$ satisfying $\eta\equiv1$ in $B_{3/2}$, and apply the proposition to the function $\eta w$.
One finds
\[[w]_{C^{\alpha+2s}(B_{1/2})} \le C\bigl( [f]_{C^\alpha(B_1)}+[L(\eta w-w)]_{C^\alpha(B_1)}+\|w\|_{C^{\alpha'+2s}(B_2)}\bigr).\]
And since the function $\eta w-w$ vanishes in $B_{3/2}$, then we have
\begin{equation}\label{nom-h-y-t}
[L(\eta w-w)]_{C^\alpha(B_1)}\leq C[w]_{C^\alpha(\R^n)}.
\end{equation}
Thus, \eqref{pas1} follows.

We recall now the definition of the norms $\|\phi\|_{\gamma;\,U}^{(\sigma)}$; see Gilbarg-Trudinger \cite{GT}.
If $\gamma=k+\gamma'$,  with $k$ integer and $\gamma'\in (0,1]$, then
\[ [\phi]_{\gamma;U}^{(\sigma)}= \sup_{x,y\in U} \biggl(d_{x,y}^{\gamma+\sigma} \frac{|D^{k}\phi(x)-D^{k}\phi(y)|}{|x-y|^{\gamma'}}\biggr),\]
and
\[ \|\phi\|_{\gamma;U}^{(\sigma)} = \sum_{l=0}^k \sup_{x\in U} \biggl(d_x^{l+\sigma} |D^l \phi(x)|\biggr) + [\phi]_{\gamma;U}^{(\sigma)}.\]
Here, we denoted
\[d_x= \mathrm{dist}(x,\partial U) \qquad\mbox{and}\qquad d_{x,y}=\min\{d_x,d_y\}.\]

We will use next these norms.
Indeed, we can rescale \eqref{pas1} and apply it to any ball $B_\rho$ of radius $\rho>0$.
Then, dividing by $\rho^\alpha$, and taking the supremum over all the balls $B_\rho$ such that $B_{2\rho}\subset B_1$, we find
\[[w]_{\alpha+2s;B_1}^{(-\alpha)} \le C\bigl( [f]_{\alpha;B_1}^{(-\alpha+2s)} + \|w\|_{\alpha'+2s;B_1}^{(-\alpha)}+[w]_{C^\alpha(\R^n)}\bigr).\]
Thus, using that
\[\|w\|_{\gamma+2s;B_1}^{(-\alpha)}\leq \epsilon \|w\|_{\alpha+2s;B_1}^{(-\alpha)}+C(\epsilon)\|w\|_{L^\infty(B_1)}\qquad\textrm{for}\ \gamma<\alpha,\]
we deduce
\[\|w\|_{\alpha+2s;B_1}^{(-\alpha)} \le C\bigl( [f]_{\alpha;B_1}^{(-\alpha+2s)} + \|w\|_{C^\alpha(\R^n)}\bigr).\]
Moreover, since $[f]_{\alpha;B_1}^{(-\alpha+2s)}\leq \|f\|_{C^\alpha(B_1)}$,
\[\|w\|_{\alpha+2s;B_1}^{(-\alpha)} \le C\bigl( \|f\|_{C^\alpha(B_1)} + \|w\|_{C^\alpha(\R^n)}\bigr).\]
In particular, we have proved that for all $w\in C^\infty_c(\R^n)$, the following inequality holds
\[\|w\|_{C^{\alpha+2s}(B_{1/2})} \le C\bigl( \|f\|_{C^\alpha(B_1)} + \|w\|_{C^\alpha(\R^n)}\bigr).\]

Finally, by using a standard approximation argument, the result follows for any solution $u\in C^\alpha(\R^n)$, and thus we are done.
\end{proof}

We now establish the estimate with a $L^\infty$ right hand side.
As before, we prove first a preliminary result.

\begin{prop}\label{claim-a2}
Let $s\in(0,1)$, $s\neq\frac12$, and let $L$ be any operator of the form \eqref{operator-L}-\eqref{ellipt-const}.
Let $\alpha\in(0,2s)$ be such that $\lfloor 2s\rfloor<\alpha<2s$.

Let $w$ be any $C_c^\infty(\R^n)$ satisfying $L w=f$ in $B_1$, with $f\in L^\infty(B_1)$.
Then, we have the estimate
\begin{equation}\label{estw2}
[w]_{C^{2s}(B_{1/2})} \le C\bigl( \|f\|_{L^\infty(B_1)} + \|w\|_{C^{\alpha}(\R^n)}\bigr).
\end{equation}
The constant $C$ depends only on $n$, $s$, $\alpha$, and the ellipticity constants \eqref{ellipt-const}.
\end{prop}

\begin{proof}
We follow the steps of the proof of Proposition \ref{claim-a}.

Assume that the statement is false.
Then, for each integer $k\ge 0$, there exist $L_k$, $w_k$, and $f_{k}$ satisfying:
\begin{itemize}
\item $ L_kw_k =f_k $   in $B_{1}$;
\vspace{3pt}
\item $L_k$ is of the form \eqref{operator-L}-\eqref{ellipt-const};
\vspace{3pt}
\item $\|f_k\|_{L^\infty(B_1)} +\| w_k \|_{C^{\alpha}(\R^n )} \le 1$;
\item $\|w_k\|_{C^{2s}(B_{1/2})} \ge k$.
\end{itemize}
In the rest of the proof we denote
\[\nu=\lfloor 2s\rfloor,\qquad \beta=2s-\alpha.\]

Since $\nu<\alpha<2s$ we then have
\begin{equation}\label{2k22}
\sup_k \sup_ {z\in B_{1/2}} \sup_{r>0} \ r^{-\beta}\left[ w_k \right]_{C^{\alpha}(B_{r}(z))} = +\infty.
\end{equation}

Next, we define
\[ \theta(r) := \sup_k  \sup_ {z\in  B_{1/2}}  \sup_{r'>r}  (r')^{-\beta}\,\bigl[w_k\bigr]_{C^{\alpha} \left(B_{r'}(z)\right)}.\]
The function $\theta$ is monotone nonincreasing, and we have $\theta(r)<+\infty$ for $r>0$ since we are assuming that  $\|w_k\|_{C^{\alpha}(\R^n)}\le 1$.
In addition, by \eqref{2k22} we have  $\theta(r)\rightarrow +\infty$ as $r\downarrow0$.

Now, for every positive integer $m$, by definition of $\theta(1/m)$ there exist $r'_m\ge 1/m$, $k_m$, and $z_{m} \in  B_{1/2}$, for which
\begin{equation}\label{nondeg22}
(r'_m)^{-\beta} \bigl[w_{k_m}\bigr]_{C^{\alpha} \left(B_{r'_m}(z_m)\right)} \ge \frac{1}{2}\,\theta(1/m) \ge \frac{1}{2}\,\theta(r'_m).
\end{equation}
Here we have used that $\theta$ is nonincreasing.
Note that we will have $r'_m\downarrow0$.

As in the proof of Proposition \ref{claim-a}, we define $p_{k,z,r}(\cdot\,-z)$ as the polynomial of degree less or equal than $\nu$ in the variables $(x-z)$ which best fits $u_k$ in $B_r(z)$ by least squares, and we denote $p_m = p_{k_m, z_m,r'_m}$.

We consider the blow up sequence
\begin{equation}\label{eqvm2}
 v_m(x) = \frac{w_{k_m}(z_{m} +r'_m x)-p_{m}(r'_m x)}{(r'_m)^{\alpha+\beta}\theta(r'_m)}.
 \end{equation}
Note that, for all $m\ge 1$ we have
\begin{equation}\label{22}
\int_{B_1(0)} v_m(x) q(x) \,dx =0\quad \mbox{for all } q\in \mathbb P_\nu.
\end{equation}
Note also that \eqref{nondeg22} implies the following nondegeneracy condition for all $m\geq1$:
\begin{equation}\label{nondeg352}
[v_m]_{C^{\alpha}(B_1)}\ge 1/2.
\end{equation}

Next, as in \eqref{growthc0}, one can show that
\begin{equation}\label{growthc02}
[v_{m}]_{C^{\alpha} (B_R)} \leq CR^{\beta}\quad \textrm{for all}\ \,R\ge 1.
\end{equation}

When $R=1$, \eqref{growthc02} implies that $\|v_m- q\|_{L^\infty(B_1)}\le C$, for some $q\in \mathbb P_\nu$.
Therefore, \eqref{22} yields
\begin{equation}\label{boundedinB12}
\|v_m\|_{L^\infty(B_1)}\le C.
\end{equation}

We now prove the following:

\vspace{6pt}
\noindent {\em Claim I.} Given $\epsilon>0$ small, the sequence $v_m$ converges in $C^{\alpha-\epsilon}_{\rm loc}(\R^n)$ to a function $v\in C^{\alpha}_{\rm loc}(\R^n)$. This function $v$ satisfies the assumptions of the Liouville-type Theorem~\ref{Liouv-entire}.
\vspace{6pt}

The $C^{\alpha-\epsilon}$ uniform convergence on compact sets of $\R^n$ of the function $v_m$ to some $v\in C^{\alpha}(\R^n)$ follows from \eqref{growthc02} and the Arzel\`a-Ascoli theorem.
Moreover, passing to the limit \eqref{growthc02}, we find that
\begin{equation}\label{growth-limit-fnct-v}
[v]_{C^{\alpha}(B_R)}\leq CR^{\beta}\qquad\textrm{for all}\ R\geq1.
\end{equation}
Thus, $v$ satisfies the growth assumption in Theorem \ref{Liouv-entire}.

On the hand, each $w_k$ satisfies a $L_kw_k =f_k$ in $B_1$.
Thus,  recalling that we have $\|f_k\|_{L^\infty(B_1)}\le 1$, we find that
\begin{equation}\label{uuu2}
\bigl|L_k w_k(\bar x+ \bar h)- L_k w_k(\bar x)\bigr|\leq 2\qquad
\textrm{for all}\ \bar x\in B_{1/2}(z)\ \textrm{and}\ \bar h\in B_{1/2}.
\end{equation}

Next, as is \eqref{vv}, one can translate \eqref{uuu2} from $w_{k_m}$ to $v_{m}$.
Indeed, setting $\bar h= r'_m h$,  and $\bar x=z_m+ r'_mx$ in \eqref{uuu2}, one has
\begin{equation}\label{vv2}
\left| L_{k_m} \left( v_{m}(\,\cdot\,+ h)-v_{m}\right)(x)\right|\leq \frac {2}{\theta(r'_m)} \quad \mbox{whenever } |x| \le \frac{1}{2 r'_m}.
\end{equation}

By Lemma \ref{lem-subseq}, the operators $L_{k_m}$ converge weakly (up to subsequence) to an operator $L$.
Thus, passing \eqref{vv2} to the limit we find that
\[L \left(  v(\,\cdot\,+ h)-v\right)=0 \quad \mbox{in all of }\R^n.\]

Notice that to be able to pass to the limit $m \to +\infty$ on \eqref{vv2} we used that, by \eqref{growthc02}, the functions
$v_{k_m}(\,\cdot\,+ h)-v_{k_m}$ satisfy the growth control
\[\|v_{k_m}(\,\cdot\,+ h)-v_{k_m}\|_{L^\infty(B_R)}\leq C|h|^\alpha R^{\beta}\qquad\textrm{for all}\ R\geq1,\]
and we are also using that $L_{k_m}$ converge weakly to $L$ as $m\rightarrow\infty$.

This finishes the proof of Claim.
\vspace{5pt}

We have thus proved that the limit function $v$ satisfies the assumptions of Theorem \ref{Liouv-entire}, and hence we conclude that $v$ is a polynomial of degree $\nu$.
On the other hand, passing \eqref{22} to the limit we obtain that $v$ is orthogonal to every polynomial of degree $\nu$ in $B_1$, and hence it must be $v\equiv 0$.
But then passing \eqref{nondeg352} to the limit we obtain that $v$ cannot be constantly zero in $B_1$; a contradiction.
\end{proof}

We also have the following.

\begin{prop}\label{claim-a22}
Let $s=\frac12$, and let $L$ be any operator of the form \eqref{operator-L}-\eqref{ellipt-const}.
Let $\alpha\in(0,2)$ be such that $\lfloor 2s\rfloor<\alpha<2s$.

Let $w$ be any $C_c^\infty(\R^n)$ satisfying $L w=f$ in $B_1$, with $f\in L^\infty(B_1)$.
Then, we have the estimate
\[[w]_{C^{2s-\epsilon}(B_{1/2})} \le C\bigl( [f]_{L^\infty(B_1)} + \|w\|_{C^{\alpha}(\R^n)}\bigr).\]
The constant $C$ depends only on $n$, $s$, $\alpha$, and the ellipticity constants \eqref{ellipt-const}.
\end{prop}

\begin{proof}
The proof is minor modification of the one in Proposition \ref{claim-a2}.
One only has to take $\beta=2s-\alpha-\epsilon$ instead of $\beta=2s-\alpha$, and follow the same steps as in Proposition \ref{claim-a2}.
\end{proof}

Finally, we can give the:

\begin{proof}[Proof of Theorem \ref{thm-interior-ball} (a)]
We prove only the case $s\neq\frac12$, the case $s=\frac12$ follows with exactly the same argument.

By Proposition \ref{claim-a2}, for all $w\in C^\infty_c(\R^n)$ we have the estimate
\[[w]_{C^{2s}(B_{1/2})} \le C\bigl( \|f\|_{L^\infty(B_1)} + \|w\|_{C^{\alpha}(\R^n)}\bigr),\]
where $\alpha$ is such that $\lfloor 2s\rfloor < \alpha < 2s$.

Then, multiplying $w$ by a cutoff function, it immediately follows that
\begin{equation}\label{pas12}
[w]_{C^{2s}(B_{1/2})} \le C\bigl( \|f\|_{L^\infty(B_1)} + \|w\|_{C^{\alpha}(B_2)}+\|w\|_{L^\infty(\R^n)}\bigr);
\end{equation}
see the proof of Theorem \ref{thm-interior-ball} (b) above.

Now, using the norms $\|\phi\|_{\gamma;\,U}^{(\sigma)}$ defined before, we can rescale \eqref{pas12} and apply it to any ball $B_\rho$ of radius $\rho>0$.
Then, taking the supremum over all the balls $B_\rho$ such that $B_{2\rho}\subset B_1$, we find
\[[w]_{2s;B_1}^{(0)} \le C\bigl( \|f\|_{0;B_1}^{(2s)} + \|w\|_{\alpha;B_1}^{(0)}+\|w\|_{L^\infty(\R^n)}\bigr).\]
Thus, we deduce
\[\|w\|_{2s;B_1}^{(0)} \le C\bigl( \|f\|_{0;B_1}^{(2s)} + \|w\|_{L^\infty(\R^n)}\bigr).\]
In particular, for all $w\in C^\infty_c(\R^n)$, the following inequality holds
\[\|w\|_{C^{2s}(B_{1/2})} \le C\bigl( \|f\|_{L^\infty(B_1)} + \|w\|_{L^\infty(\R^n)}\bigr).\]

Finally, by using a standard approximation argument, the result follows.
\end{proof}

To end this section, we give an immediate consequence of Theorem \ref{thm-interior-ball}.
Notice that here we assume some regularity on the spectral measure $a$, but the ellipticity constants are the same as before.
In particular, we are not assuming positivity of $a$ in all of $S^{n-1}$.

\begin{cor}\label{cor-interior-ball}
Let $s\in(0,1)$, $L$ be given by \eqref{operator-L2}, and assume that
\[a\in C^\alpha(S^{n-1}).\]
Let $u$ be a solution of \eqref{eq-ball}.
Then, if $f\in C^\alpha(B_1)$ and $u\in L^\infty(\R^n)$,
\[\|u\|_{C^{\alpha+2s}(B_{1/2})}\leq C\left(\|u\|_{L^\infty(\R^n)}+\|f\|_{C^\alpha(B_1)}\right)\]
whenever $\alpha+2s$ is not an integer.

The constant $C$ depends only on $n$, $s$, ellipticity constants \eqref{ellipt-const}, and $\|a\|_{C^\alpha(S^{n-1})}$.
\end{cor}

\begin{proof}
The proof is a minor modification of the proof of Theorem \ref{thm-interior-ball} (b).
Indeed, one only needs to replace the estimate \eqref{nom-h-y-t} therein, by the following one
\[[L(\eta w-w)]_{C^\alpha(B_1)}\leq C[w]_{L^\infty(\R^n)},\]
which follows easily using that $a\in C^\alpha(S^{n-1})$ ---recall that $\eta\equiv1$ in $B_1$ and $\eta\in C^\infty_c(B_2)$.
With this modification, the rest of the proof is exactly the same.
\end{proof}

Finally, we give an immediate consequence of Theorem \ref{thm-interior-ball} that will be used later.

\begin{cor}\label{cor-interior-ball-growth}
Let $s\in(0,1)$, and let $L$ be any operator of the form \eqref{operator-L}-\eqref{ellipt-const}.
Let $u$ be any solution of
\[L u =f\quad {in}\ B_1,\]
with $f\in L^\infty(B_1)$.
Then, for any $\epsilon>0$,
\[\|u\|_{C^{2s}(B_{1/2})}\leq C\left(\sup_{R\geq1}R^{\epsilon-2s}\|u\|_{L^\infty(B_R)}+\|f\|_{L^\infty(B_1)}\right)\quad \textrm{if}\ s\neq\frac12,\]
and
\[\|u\|_{C^{2s-\epsilon}(B_{1/2})}\leq C\left(\sup_{R\geq1}R^{\epsilon-2s}\|u\|_{L^\infty(B_R)}+\|f\|_{L^\infty(B_1)}\right)\quad \textrm{if}\ s=\frac12,\]
The constant $C$ depends only on $n$, $s$, $\epsilon$, and the ellipticity constants \eqref{ellipt-const}.
\end{cor}

\begin{proof}
The proof follows by using that the truncated function $\tilde u=u\chi_{B_2}$ satisfies the hypotheses of Theorem \ref{thm-interior-ball}.
\end{proof}

\section{A Liouville theorem in the half space}
\label{sec5}

In this Section we prove the following Liouville-type theorem, which will be needed in the proof of Theorem \ref{thm-bdry-reg}.

\begin{thm}\label{Liouv-half}
Let $L$ be an operator of the form \eqref{operator-L}-\eqref{ellipt-const}.
Let $u$ be any weak solution of
\begin{equation}\label{eq-I-flat}
\left\{ \begin{array}{rcll}
L u &=&0&\textrm{in }\R^n_+ \\
u&=&0&\textrm{in }\R^n_-.
\end{array}\right.
\end{equation}
Assume that, for some $\beta<2s$, $u$ satisfies the growth control
\[\|u\|_{L^\infty(B_R)}\leq CR^{\beta}\quad \textrm{for all}\ R\geq1.\]
Then,
\[u(x)=K(x_n)_+^s\]
for some constant $K\in \R$.
\end{thm}

Notice that Theorem \ref{Liouv-half} is related to Theorem 1.4 in \cite{RS-K}.
However, the proofs of the two results are quite different.
Indeed, in \cite{RS-K} we first used a method of Caffarelli to obtain a H\"older estimate for $u/d^s$ up to the boundary, and then we iterated this estimate to show the Liouville theorem.
Here, instead, we only use estimates for $u$ (and not for $u/d^s$) to establish Theorem \ref{Liouv-half}.

Recall that in the present context we can not use the method of Caffarelli (that we adapted to nonlocal equations in \cite{RS-K,RS-Dir}), because the operators \eqref{operator-L}-\eqref{ellipt-const} do not satisfy a Harnack inequality.

\subsection{Barriers}
\label{sec-barriers}

We next construct supersolutions and subsolutions that are needed in our analysis.
We will need them both in the proofs of the Liouville Theorem \ref{Liouv-half} and of Theorem \ref{thm-bdry-reg}.

These barriers are essentially the same as the ones constructed in our work \cite{RS-K}, however the proofs must be redone so that the ellipticity constants are \eqref{ellipt-const}.

Before constructing the sub and supersolution, we give two preliminary lemmas.
These are the analogues of Lemmas 3.1 and 3.2 in \cite{RS-K}.

\begin{lem}\label{lem11}
Let $s\in(0,1)$, and let $L$ be given by \eqref{operator-L}-\eqref{ellipt-const}.
Let
 \[\varphi^{(1)}(x) = \bigl({\rm dist}(x,B_1)\bigr)^s\quad\mbox{and}\quad \varphi^{(2)}(x) = \bigl({\rm dist}(x,\R^n\setminus B_1)\bigr)^s.\]
 Then,
\begin{equation}\label{eqvarphi1}
0 \le  L \varphi^{(1)}(x) \le L \varphi^{(1)}(x) \le  C\left\{1+\bigl|\log(|x|-1)\bigr|\right\}  \quad \mbox{in } B_{2}\setminus B_1.
\end{equation}
and
\begin{equation}\label{eqvarphi2}
0\ge L \varphi^{(2)}(x)\ge L \varphi^{(2)}(x) \ge   -C\left\{1+\bigl|\log(1-|x|)\bigr|\right\} \ \mbox{in }B_{1}\setminus B_{1/2}.
\end{equation}
The constant $C$ depends only on $s$, $n$, and the ellipticity constants \eqref{ellipt-const}.
\end{lem}

\begin{proof}[Proof of Lemma \ref{lem11}]
We use the notation $x=(x',x_n)$ with $x'\in \R^{n-1}$.
To prove \eqref{eqvarphi1} let us estimate $L \varphi^{(1)}(x_\rho) $ where $x_\rho = (0,1+\rho)$ for $\rho\in(0,1)$.
To do it, we subtract the function $\psi(x)=(x_n-1)_+^s$, which satisfies $L\psi(x_\rho)=0$.
As in \cite[Lemma 3.1]{RS-K}, we have that
\[0\le \bigl(\varphi^{(1)}_1-\psi\bigr)(x_\rho+ y) \le
\begin{cases}
C \rho^{s-1} |y'|^2 \quad \mbox{for } y=(y',y_n)\in B_{\rho/2}\\
C |y'|^{2s}         \quad \mbox{for } y=(y',y_n)\in B_1\setminus B_{\rho/2} \\
C |y|^s              \quad \mbox{for } y\in \R^n \setminus B_1.
\end{cases}
\]

Therefore,
\[
\begin{split}
0&\le L\varphi^{(1)}(x_\rho) =  L\bigl(\varphi^{(1)}- \psi\bigr)(x_\rho)
\\
&=\int_{S^{n-1}}\int_{-\infty}^{+\infty} \frac{\bigl(\varphi^{(1)}_1-\psi\bigr)(x_\rho+ r\theta) + \bigl(\varphi^{(1)}_1-\psi\bigr)(x_\rho-r\theta)}{2} \frac{dr}{|r|^{1+2s}}\,d\mu(\theta)
\\
&\le  C\int_{S^{n-1}} \left(\int_{|r|<\rho/2}\frac{ \rho^{s-1}|r|^2 dr}{|r|^{1+2s}}
+\int_{\rho/2<|r|<1}\frac{|r|^{2s} dr}{|r|^{1+2s}}
+\int_{|r|>1}\frac{|r|^s dr}{|r|^{1+2s}}
\right)d\mu
\\
&\le   C\Lambda\bigl(1+|\log\rho|\bigr).
\end{split}
\]
Thus, \eqref{eqvarphi1} follows.
Finally, \eqref{eqvarphi2} follows with a similar argument.
\end{proof}

\begin{lem}\label{lem12}
Let $s\in(0,1)$, and let $L$ be given by \eqref{operator-L}-\eqref{ellipt-const}.
Let
 \[\varphi^{(3)}(x) = \bigl({\rm dist}(x,B_1)\bigr)^{3s/2}\quad\mbox{and}\quad \varphi^{(4)}(x) = \bigl({\rm dist}(x,\R^n\setminus B_1)\bigr)^{3s/2}.\]
 Then,
\begin{equation}\label{eqvarphi3}
L \varphi^{(3)}(x) \ge  c(|x|-1)^{-s/2}  \quad \mbox{for all  }x\in B_{2}\setminus B_1.
\end{equation}
and
\begin{equation}\label{eqvarphi4}
L \varphi^{(4)}(x) \ge  c(1-|x|)^{-s/2}- C  \quad \mbox{for all  }x\in B_{1}\setminus B_{1/2}.
\end{equation}
The constants $c>0$ and $C$ depend only on $n$, $s$, and the ellipticity constants \eqref{ellipt-const}.
\end{lem}

\begin{proof}
As before, we denote $x=(x',x_n)$ with $x'\in \R^{n-1}$.
To prove \eqref{eqvarphi4} let us estimate $L\varphi^{(4)} (x_\rho)$, where $x_\rho=(0,1+\rho)$ for $\rho\in(0,1)$.
To do it we subtract the function $\psi(x)= (1-x_n)_+^{3s/2}$.
It can be seen, as in \cite{RS-K}, that $\psi$ satisfies $L\psi(x_\rho)= c\rho^{-s/2}$ for some $c>0$.
We note that
\[\bigl(\varphi^{(4)}-\psi\bigr)(x_\rho) = 0\]
and, as in \cite[Lemma 3.2]{RS-K},
\[0\ge\bigl(\varphi^{(4)}-\psi\bigr)(x_\rho+ y) \ge
\begin{cases}
- C \rho^{3s/2-1} |y'|^2 \quad \mbox{for } y=(y',y_n)\in B_{\rho/2}\\
- C |y'|^{3s}         \quad \mbox{for } y=(y',y_n)\in B_1\setminus B_{\rho/2} \\
- C |y|^{3s/2}              \quad \mbox{for } y\in \R^n \setminus B_1.
\end{cases}
\]
Then, one finds that
\[
L\varphi^{(4)}(x_\rho) -c\rho^{-s/2} \ge  - C,
\]
which establishes \eqref{eqvarphi4}.
The estimate \eqref{eqvarphi3} follows similarly.
\end{proof}

Using the previous lemmas, one can now construct the sub and supersolutions that will be used in the next section.
We skip the proofs of the following two lemmas because they are exactly the same as the ones given in \cite[Lemmas 3.3 and 3.4]{RS-K}.

\begin{lem}[Supersolution]\label{supersol}
Let $s\in(0,1)$, and let $L$ be given by \eqref{operator-L}-\eqref{ellipt-const}.
There are positive constants $\epsilon$ and $C$, and a radial, bounded, continuous function $\varphi_1$ which is $C^{1,1}$ in $B_{1+\epsilon}\setminus\overline B_1$ and satisfies
\[
\begin{cases}
L\varphi_1(x) \le -1   &\mbox{ in } B_{1+\epsilon}\setminus \overline{B_1} \\
\varphi_1(x) = 0 \quad &\mbox{ in } B_1\\
\varphi_1(x) \le C\bigl(|x|-1\bigr)^s  &\mbox{ in } \R^n\setminus B_1\\
\varphi_1(x) \ge 1  &\mbox{ in } \R^n\setminus B_{1+\epsilon}
\end{cases}
\]
The constants $\epsilon$, $c$ and $C$ depend only on $n$, $s$, and ellipticity constants.
\end{lem}

\begin{proof}
See the proof of Lemma 3.3 in \cite{RS-K}.
\end{proof}

\begin{lem}[Subsolution]\label{subsol}
Let $s\in(0,1)$, and let $L$ be given by \eqref{operator-L}-\eqref{ellipt-const}.
There is $c>0$, and a radial, bounded, continuous function $\varphi_2$ that satisfies
\[
\begin{cases}
L\varphi_2(x) \ge c  &\mbox{ in } B_1\setminus B_{1/2}\\
\varphi_2(x) = 0 \quad &\mbox{ in } \R^n\setminus B_1\\
\varphi_2(x) \ge c\bigl(1-|x|\bigr)^s  &\mbox{ in } B_1\\
\varphi_2(x) \le 1  &\mbox{ in } \overline{B_{1/2}}.
\end{cases}
\]
The constants $\epsilon$, $c$ and $C$ depend only on $n$, $s$, and ellipticity constants.
\end{lem}

\begin{proof}
See the proof of Lemma 3.4 in \cite{RS-K}.
\end{proof}

\subsection{H\"older regularity up to the boundary for $u$}

Using the interior estimates and the supersolution constructed above, we find the following.

\begin{prop}\label{prop-Cs}
Let $s\in(0,1)$, $L$ be any operator of the form \eqref{operator-L}-\eqref{ellipt-const}, and $\Omega$ be any bounded Lipschitz domain satisfying the exterior ball condition.
Let $f\in L^\infty(\Omega)$, and $u$ be a weak solution of
\begin{equation}\label{eq-domain-Cs}
\left\{ \begin{array}{rcll}
L u &=&f&\textrm{in }\Omega \\
u&=&0&\textrm{in }\R^n\setminus\Omega.
\end{array}\right.
\end{equation}
Then,
\[\|u\|_{C^s(\overline\Omega)}\leq C\|f\|_{L^\infty(\Omega)}.\]
The constant $C$ depends only on $n$, $s$, $\Omega$, and the ellipticity constants \eqref{ellipt-const}.
\end{prop}

\begin{proof}
The proof of this result is quite standard once one has interior estimates (given by Theorem \ref{thm-interior-ball}) and an appropriate barrier (given by Lemma \ref{supersol}).
For more details, see the proof of Proposition 1.1 in \cite{RS-Dir}, where this was done for the case of the fractional Laplacian.
\end{proof}

We will also need the following version of the estimate.

\begin{prop}\label{prop-Cs-2}
Let $L$ be any operator of the form \eqref{operator-L}-\eqref{ellipt-const}.
Let $f\in L^\infty(B_1^+)$, and $u\in L^\infty(\R^n)$ be any bounded solution of
\begin{equation}\label{eq-567}
\left\{ \begin{array}{rcll}
L u &=&f&\textrm{in }B_1^+ \\
u&=&0&\textrm{in }B_1^-.
\end{array}\right.
\end{equation}
Then,
\[\|u\|_{C^s(\overline {B_{1/2}})}\leq C\bigl(\|f\|_{L^\infty(B_1)}+\|u\|_{L^\infty(\R^n)}\bigr).\]
\end{prop}

\begin{proof}
It follows from the previous result by multiplying $u$ by a cutoff function.
\end{proof}

\subsection{Proof of Theorem \ref{Liouv-half}}

Here we prove Theorem \ref{Liouv-half}.
For it, we will need the following result, established in \cite{RS-K}.

\begin{lem}[\cite{RS-K}]\label{classification1D}
Let $u$ satisfy $(-\Delta)^s u= 0$ in $\R_+$ and $u=0$ in $\R_-$.
Assume that, for some $\beta\in (0,2s)$, $u$ satisfies the growth control $\|u\|_{L^\infty(0,R)}\le CR^\beta$ for all $R\geq1$.
Then $u(x) = K(x_+)^s$.
\end{lem}

We can now give the:

\begin{proof}[Proof of Theorem \ref{Liouv-half}]
Given $\rho\geq1$, define $v_\rho(x)=\rho^{-\beta}u(\rho x)$.
Then, it follows that $v_\rho$ satisfies the same growth condition as $u$, namely
\[\|v_\rho\|_{L^\infty(B_R)}\leq CR^{\beta}\quad \textrm{for all}\ R\geq1.\]
Indeed, one has
\[\|v_\rho\|_{L^\infty(B_R)}=\rho^{-\beta}\|u\|_{L^\infty(B_{\rho R})}\leq \rho^{-\beta}\, C(\rho R)^\beta= CR^\beta.\]
Moreover, we know that  $Lv_\rho=0$ in $\R^n_+$ and $v_\rho=0$ in $\R^n_-$.

In particular, if we consider $\bar v_\rho=v_\rho\chi_{B_2}$, then $\bar v_\rho\in L^\infty(\R^n)$ satisfies
\begin{equation}\label{eq-567}
\left\{ \begin{array}{rcll}
L\bar v_\rho  &= &\bar g&\textrm{in }B_1^+ \\
\bar v_\rho&=&0&\textrm{in }B_1^-
\end{array}\right.
\end{equation}
for some $g\in L^\infty(B_1^+)$.
Denote $C_0=\|g\|_{L^\infty(B_1^+)}$.
Then, by Proposition \ref{prop-Cs-2}, it follows that
\[\|v_\rho\|_{C^s(B_{1/2})}=\|\bar v_\rho\|_{C^s(B_{1/2})}\leq CC_0.\]
Therefore, we find
\[[u]_{C^s(B_{\rho/2})}=\rho^{-s}[u(\rho x)]_{C^s(B_{1/2})}=\rho^{\beta-s}[v_\rho]_{C^s(B_{1/2})}\leq CC_0\rho^{\beta-s}.\]
In other words, we have proved that
\[[u]_{C^s(B_R)}\leq CR^{\beta-s}\qquad \textrm{for all}\ R\geq1.\]

Now, given $\tau\in S^{n-1}$ such that $\tau_n=0$, and given $h>0$, consider
\[w(x)=\frac{u(x+h\tau)-u(x)}{h^s}.\]
By the previous considerations, we have
\[\|w\|_{L^\infty(B_R)}\leq CR^{\beta-s}\qquad \textrm{for all}\ R\geq1.\]
Moreover, we clearly have $Lw=0$ in $\R^n_+$ and $w=0$ in $\R^n_-$.
Therefore, we can repeat the previous argument (applied to $w$ instead of $u$), to find that
\[[w]_{C^s(B_R)}\leq CR^{\beta-2s}\qquad \textrm{for all}\ R\geq1.\]

Hence, since $\beta<2s$, letting $R\rightarrow\infty$ in the previous inequality we find that
\[w\equiv0\qquad \textrm{in}\ \,\R^n.\]
Therefore, $u(x+h\tau)=u(x)$ for all $h>0$ and for all $\tau\in S^{n-1}$ such that $\tau_n=0$.
Thus, we have that $u$ depends only on the $x_n$-variable, i.e.,
\[u(x)=\bar u(x_n)\]
for some 1D function $\bar u$.

But we then have that
\[\begin{split}
Lu(x)&=\int_{S^{n-1}}\int_{-\infty}^{+\infty}\bigl(u(x+\theta r)+u(x-\theta r)-2u(x)\bigr)\frac{dr}{|r|^{1+2s}}\,d\mu(\theta)\\
&=\int_{S^{n-1}}\int_{-\infty}^{+\infty}\bigl(\bar u(x_n+\theta_n r)+\bar u(x_n-\theta_n r)-2\bar u(x_n)\bigr)\frac{dr}{|r|^{1+2s}}\,d\mu(\theta)\\
&=-c\int_{S^{n-1}}\int_{-\infty}^{+\infty}(-\Delta)^s_{\R}\bar u(x_n+\theta_n \cdot)\,d\mu(\theta)\\
&=-c\int_{S^{n-1}}\int_{-\infty}^{+\infty}|\theta_n|^{2s}(-\Delta)^s_{\R}\bar u(x_n)\,d\mu(\theta)\\
&=-c\,(-\Delta)^s\bar u(x_n),\end{split}\]
for some constant $c>0$.
Therefore, $\bar u$ solves $(-\Delta)^s\bar u=0$ in $\R_+$, $\bar u=0$ in $\R_-$.
Hence, using Lemma \ref{classification1D} we finally deduce that $\bar u(x_n)=K(x_n)_+^s$, and thus
\[u(x)=K(x_n)_+^s,\]
as desired.
\end{proof}

\section{Boundary regularity}
\label{sec6}

In this section we finally prove Theorem \ref{thm-bdry-reg}.

The main ingredient in its proof is the following result.
In it, we use the following terminology.

\begin{defi}\label{defi-split}
We say that $\Gamma$ is a $C^{1,1}$ surface with radius $\rho_0>0$ splitting $B_1$ into $U^+$ and $U^-$ if the following happens:
\begin{itemize}
\item The two disjoint domains $U^+$ and $U^-$ partition $B_1$, i.e., $\overline{B_1}=\overline{U^+}\cup \overline{U^-}$.
\item The boundary $\Gamma:=\partial U^+\backslash \partial B_1=\partial U^-\backslash \partial B_1$ is a $C^{1,1}$ surface with $0\in\Gamma$.
\item All points on $\Gamma\cap \overline{B_{3/4}}$ can be touched by two balls of radii $\rho_0$, one contained in $U^+$ and the other contained in $U^-$.
\end{itemize}
\end{defi}

The result reads as follows.

\begin{prop}\label{proponbdryreg}
Let $s\in (0,1)$ and $\beta\in(s,2s)$ be given constants.
Let $\Gamma$ be a $C^{1,1}$ surface with radius $\rho_0$ splitting $B_1$ into $U^+$ and $U^-$; see Definition~\ref{defi-split}.

Let $f\in L^\infty(U^+)$, and assume that $u\in L^\infty(\R^n)$ is a solution of
\[\left\{\begin{array}{rcl}
L u &= & f \quad \mbox{in } U^+\\
u  &=&0\quad \mbox{in }U^-,
\end{array}\right.\]
where $L$ is any operator of the form \eqref{operator-L}-\eqref{ellipt-const}.

Then, for all $z\in \Gamma\cap \overline{B_{1/2}}$ there is a constant $Q(z)$ with $|Q(z)|\leq CC_0$ for which
\[\left|u(x)-Q(z)\bigl((x-z)\cdot\nu(z)\bigr)_+^s\right|\leq CC_0|x-z|^\beta\qquad \textrm{for all}\ x\in B_1,\]
where $\nu(z)$ is the unit normal vector to $\Gamma$ at $z$ pointing towards $U^+$ and
\[C_0=\|u\|_{L^\infty(\R^n)}+\|f\|_{L^\infty(U^+)}.\]
The constant $C$ depends only on $n$, $\rho_0$, $s$, $\beta$, and the ellipticity constants \eqref{ellipt-const}.
\end{prop}

In order to show Proposition \ref{proponbdryreg}, we will need some preliminary lemmas.

First, we will need the following technical result.

\begin{lem}\label{lemseries}
Let $\beta>s$ and $\nu\in S^{n-1}$ be some unit vector.
Let $u\in C(B_1)$ and define
\begin{equation}\label{defphir}
\phi_r(x) := Q_*(r )\,(x\cdot \nu)_+^s,
\end{equation}
where
\[Q_*(r ) : = {\rm arg\,min}_{Q\in \R} \int_{B_r} \bigl(u(x) - Q(x\cdot \nu)_+^s\bigr)^2\,dx = \frac{\int_{B_r} u(x) \, (x\cdot \nu)_+^s \,dx}{\int_{B_r} (x\cdot \nu)_+^{2s} \,dx}.\]

Assume that for all $r\in(0,1)$ we have
\begin{equation}\label{allr}
\bigl\| u - \phi_r\bigr\|_{L^\infty(B_r)} \le C_0 r^\beta.
\end{equation}
Then, there is $Q\in \R$ satisfying $|Q|\le C\bigl(C_0+\|u\|_{L^\infty(B_1)}\bigr)$ such that
\[ \bigl\|u - Q(x\cdot \nu)_+^s \bigr\|_{L^\infty(B_r)}\le C C_0 r^{\beta}\]
for some constant $C$ depending only on $\beta$ and $s$.
\end{lem}

\begin{proof}
We may assume $\|u\|_{L^\infty(B_1)}=1$.
By \eqref{allr}, for all $x' \in  B_r$ we have
\[\bigl| \phi_{2r}(x') - \phi_{r}(x')\bigr| \le  \bigl| u(x') -\phi_{2r}(x')\bigr| + \bigl| u(x') -\phi_{r}(x')\bigr| \le C C_0 r^{\beta}.\]
But this happening for every $x' \in  B_r$ yields, recalling \eqref{defphir},
\[ \bigl| Q_*(2r) - Q_*(r ) \bigr| \le C C_0 r^{\beta-s}.\]
In addition, since $\|u\|_{L^\infty(B_1)}= 1$, we clearly have that
\begin{equation}\label{boundQstar}
| Q_*(1)|\le C.
\end{equation}
Since $\beta>s$, this implies the existence of the limit
\[ Q := \lim_{r\searrow 0} Q_*(r ).\]
Moreover, using again $\beta-s>0$,
\[\bigl| Q -Q_*(r ) \bigr| \le  \sum_{m=0}^\infty \bigl| Q_*(2^{-m}r) - Q_*(2^{-m-1}r) \bigr|
\le  \sum_{m=0}^\infty C C_0 2^{-m(\beta-s)}  r^{\beta-s} \le C C_0 r^{\beta-s}.\]
In particular, using \eqref{boundQstar} we obtain
\begin{equation}\label{boundab2}
|Q|\le C(C_0+1) .
\end{equation}

We have thus proven that for all $r\in(0,1)$
\[
\begin{split}
\bigl\| u - Q(x\cdot\nu)_+^s\|_{L^\infty(B_r)}  &\le  \| u- Q_* (r )(x\cdot\nu)_+^s\|_{L^\infty(B_r)}  \,+\\
& \qquad \qquad +   \|Q_* (r )(x\cdot\nu)_+^s - Q(x\cdot\nu)_+^s \|_{L^\infty(B_r)}\\
&\le  C_0 r^{\beta} + |Q_*( r)-Q| r^{s} \le C(C_0+1)r^{\beta}.
\end{split}
\]
\end{proof}

Second, we will also need the following estimate in order to control the ``errors coming from the geometry of the domain''.

\begin{lem}\label{propCgammaMs}
Assume that $B_1$ is divided into two disjoint subdomains $\Omega_1$ and $\Omega_2$ such that $\overline{B_1}= \overline\Omega_1\cup \overline\Omega_2$.
Assume that  $\Gamma := \partial \Omega_1\setminus \partial B_1= \partial \Omega_2 \setminus  \partial B_1$ is a $C^{0,1}$ surface and that $0\in \Gamma$.
Moreover assume that, for some $\rho_0>0$, all the points on $\Gamma\cap \overline {B_{3/4}}$ can be touched by a ball of radius $\rho_0\in(0,1/4)$ contained in $\Omega_2$.

Let $s\in(0,1)$, and let $L$ be any operator of the form \eqref{operator-L}-\eqref{ellipt-const}.
Let $\alpha\in(0,1)$, $g\in C^\alpha\bigl(\overline{\Omega_2}\bigr)$, $f\in L^\infty(\Omega_1)$, and $u\in C(\overline{B_1})$ satisfying $|u(x)|\leq M\,(1+|x|)^\beta$ in $\R^n$ for some $\beta<2s$.
Assume that $u$ satisfies in the weak sense
\[Lu=f\ \, \textrm{in }\, \Omega_1,\quad u=g\ \textrm{ in }\,\Omega_2.\]
Then, there is $\gamma\in(0,\alpha)$ such that  $u\in C^\gamma\bigl(\overline{B_{1/2}}\bigr)$ with the estimate
\[
\|u\|_{C^\gamma(B_{1/2})}\le C \bigl(\|u\|_{L^\infty(B_1)} + \|g\|_{C^\alpha(\Omega_2)}+\|f\|_{L^\infty(\Omega_1)}+M\bigr).
\]
The constants $C$ and $\gamma$ depend only on $n$, $s$, $\alpha$, $\rho_0$, $\beta$, and ellipticity constants.
\end{lem}

\begin{proof}
Define $\tilde u= u\chi_{B_1}$.
Then $\tilde u$ satisfies $L\tilde u=\tilde f$ in $\Omega_1\cap B_{3/4}$ and $\tilde u=g$ in $\Omega_2$,
where $\|\tilde f\|_{L^\infty(\Omega_1\cap B_{3/4})} \leq C \bigl(\|f\|_{L^\infty(\Omega_1)}+M\bigr):=C_0'$.
The constant $C$ depends only on $n$, $s$, $\beta$, and ellipticity constants.

The proof consists of two steps.

\emph{First step.} We next prove that there are $\delta>0$ and $C$ such that for all $z\in \Gamma\cap\overline{B_{1/2}}$ it is
\begin{equation}\label{onbdryinproof}
\|\tilde u-g(z)\|_{L^\infty(B_r(z))}\le Cr^\delta \quad \mbox{for all }r\in(0,1),
\end{equation}
where $\delta$ and $C$ depend only on $n$, $s_0$, $C_0'$, $\|u\|_{L^\infty(B_1)}$, $\|g\|_{C^\alpha(\Omega_2)}$, and ellipticity constants.

Let  $z\in \Gamma\cap\overline{B_{1/2}}$.
By assumption, for all $R\in(0,\rho_0)$ there $y_R\in \Omega_2$ such that a ball $B_R(y_R)\subset \Omega_2$ touches $\Gamma$ at $z$, i.e., $|z-y_R|=R$.

Let $\varphi_1$ and $\epsilon>0$ be the supersolution and the constant in Lemma \ref{supersol}.
Take
\[\psi(x)= g(y_R)+  \|g\|_{C^\alpha({\Omega_2})} \bigl((1+\epsilon)R\bigr)^\alpha + \bigl(C_0'+\|u\|_{L^\infty(B_1)}\bigr)\varphi_1\left(\frac{x-y_R}{R}\right).\]
Note that $\psi$ is above $\tilde u$ in $\Omega_2\cap B_{(1+\epsilon)R}$.
On the other hand, from the properties of $\varphi_1$, it is
$M^+ \psi \le -\bigl(C_0'+\|u\|_{L^\infty(B_1)}\bigr) R^{-2s} \le -C_0'$  in the annulus $B_{(1+\epsilon)R}(y_R)\setminus B_R(y_R)$, while $\psi\ge \|u\|_{L^\infty(B_1)} \ge \tilde u$ outside $B_{(1+\epsilon)R}(y_R)$.
It follows that $\tilde u \le \psi$ and thus we have
\[ \tilde u(x)- g(z) \le C\bigl(R^\alpha+ (r/R)^s\bigl) \quad \mbox{for all }x\in B_r(z)\quad\textrm{and for all}\ r\in(0,\epsilon R)\ \textrm{and}\ R\in(0,\rho_0).\]
Here, $C$ denotes a constant depending only on $n$, $s_0$, $C_0'$, $\|u\|_{L^\infty(B_1)}$, $\|g\|_{C^\alpha({\Omega_2})}$, and ellipticity constants.
Taking $R=r^{1/2}$ and repeating the argument up-side down we obtain
\[ |\tilde u(x)- g(z)| \le C\bigl(r^{\alpha/2}+r^{s/2}\bigr)\leq Cr^\delta \quad \mbox{for all }x\in B_r(z)\ \textrm{and}\ r\in(0,\epsilon^{1/2})\]
for $\delta=\frac12\min\{\alpha,s_0\}$.
Taking a larger constant $C$, \eqref{onbdryinproof} follows.

\emph{Second step.} We now show that \eqref{onbdryinproof} and the interior estimates in Theorem \ref{thm-interior-ball} (b) imply $\|u\|_{C^\gamma(B_{1/2})}\le C$, where $C$ depends only on the same quantities as above.

Indeed, given $x_0\in \Omega_1\cap B_{1/2}$, let $z\in \Gamma$ and $r>0$ be such that
\[d= {\rm dist\,}(x_0, \Gamma)= {\rm dist\,}(x_0,z).\]
Let us consider
\[ v(x) = \tilde u\left(x_0+\frac{d}{2} x\right)- g(z).\]
We clearly have
\[ \|v\|_{L^\infty(B_1)}\le C\quad \mbox{and}\quad \|v\|_{L^\infty(B_R)}\le CR^s\quad\textrm{for}\ R\geq1.\]
On the other hand, $v$ satisfies
\[ Lv(x) =(d/2)^{2s} L\tilde u(x_0+rx)\quad \textrm{in}\ B_1\]
and thus
\[ \|Lv\|_{L^\infty(B_1)}\leq C_0'\quad \textrm{in}\ B_1.\]
Therefore, Corollary \ref{cor-interior-ball-growth} yields
\[ \|v\|_{C^\alpha(B_{1/2})} \le C\]
or equivalently
\begin{equation}\label{interiorinproof}
[u]_{C^\alpha(B_{d/4}(x_0))}\le Cd^{-\alpha}.
\end{equation}

Combining \eqref{onbdryinproof} and \eqref{interiorinproof}, using the same argument as in the proof of Proposition 1.1 in \cite{RS-K}, we obtain
\[ \|u\|_{C^\gamma(\Omega_1\cap B_{1/2})}\le C,\]
as desired.
\end{proof}

Using the previous results, and a compactness argument in the spirit of the one in \cite{RS-K}, we can give the:

\begin{proof}[Proof of Proposition \ref{proponbdryreg}]
Assume that there are sequences $\Gamma_k$, $\Omega^+_k$, $\Omega^-_k$, $f_k$, $u_k$, and $L_k$ that satisfy the assumptions of the proposition, that is,
\begin{itemize}
\item $\Gamma_k$ is a $C^{1,1}$ hyper surface with radius $\rho_0$ splitting $B_1$ into $\Omega_k^+$ and $\Omega_k^-$;
\item $L_k$ is of the form \eqref{operator-L} and satisfying \eqref{ellipt-const};
\item $\|u_k\|_{L^\infty(\R^n)} +\|f_k\|_{L^\infty(\Omega^+_k)} =1$;
\item $u_k$ is a solution of $L u_k = f_k$ in $\Omega^+_k$ and $u_k= 0$ in $\Omega^-_k$;
\end{itemize}
but suppose for a contradiction that the conclusion of the proposition does not hold.
That is, for all $C>0$, there are $k$ and $z\in \Gamma_k\cap \overline{B_{1/2}}$ for which
no constant $Q\in \R$ satisfies
\begin{equation}\label{1k}
\left| u_k(x) - Q \bigl((x-z)\cdot \nu_k(z)\bigr)_+^s \right| \le C|x-z|^{\beta}\quad \mbox{for all }x\in B_1.
\end{equation}
Here, $\nu_k(z)$ denotes the unit normal vector to $\Gamma_k$ at $z$, pointing towards $\Omega^+_k$.

In particular, using Lemma \ref{lemseries},
\begin{equation}\label{2k}
\sup_k \sup_ {z\in \Gamma_k \cap B_{1/2}} \sup_{r>0} \ r^{-\beta}\left\| u_k- \phi_{k,z,r} \right\|_{L^\infty(B_{r}(z))} = \infty,
\end{equation}
where
\begin{equation}\label{phi-k-z-r}
\phi_{k,z,r}(x) = Q_{k,z}(r)\,\bigl((x-z)\cdot \nu_k(z)\bigr)_+^s
\end{equation}
and
\[\begin{split}
Q_{k,z}(r):=& \ {\rm arg\,min}_{Q\in \R}  \int_{B_r(z)}  \left| u_k(x) -Q\bigl((x-z)\cdot \nu_k(z)\bigr)_+^s\right|^2 \,dx\\
=&\ \frac{\int_{B_r(z)}u_k(x)\bigl((x-z)\cdot\nu_k(z)\bigr)_+^s dx}{\int_{B_r(z)}\bigl((x-z)\cdot\nu_k(z)\bigr)_+^{2s}dx}.
\end{split}\]

Next define the monotone in $r$ quantity
\[\begin{split}
 \theta(r) := \sup_k  \sup_ {z\in \Gamma_k \cap B_{1/2}}  \sup_{r'>r}  \ (r')^{-\beta} &
\max\biggl\{\bigl\|u_k-\phi_{k,z,r'}\bigr\|_{L^\infty\left(B_{r'}(x_0)\right)}\, ,\\
& \hspace{25mm}(r')^s\left|Q_{k,z}(2r')-Q_{k,z}(r')\right|\biggr\}.
\end{split}\]
We have $\theta(r)<\infty$ for $r>0$ and $\theta(r)\nearrow \infty$ as $r\searrow0$.
Clearly, there are sequences $r_m\searrow 0$, $k_m$, and $z_{m}\to z \in \overline B_{1/2}$, for which
\begin{equation}\label{nondeg2-3}
\begin{split}
(r_m)^{-\beta}&\max\biggl\{\left\|u_{k_m}-\phi_{k_m, z_m,r_m}\right\|_{L^\infty(B_{r_m}(x_{m}))}\, ,\\
&\hspace{30mm}  (r_m)^s\left|Q_{k_m,z_m}(2r_m)-Q_{k_m,z_m}(r_m)\right|\biggr\}\ge \theta(r_{m})/2.
\end{split}\end{equation}
From now on in this proof we denote $\phi_m = \phi_{k_m, z_m,r_m}$ and $\nu_m= \nu_{k_m}(z_m)$.

In this situation we consider
\[ v_m(x) = \frac{u_{k_m}(z_{m} +r_m x)-\phi_{m}(z_{m}+r_m x)}{(r_m)^{\beta}\theta(r_m)}.\]
Note that, for all $m\ge 1$,
\begin{equation}\label{2-3}
\int_{B_1} v_m(x)\bigl(x\cdot \nu_m\bigr)_+^s \,dx =0.
\end{equation}
This is the optimality condition for least squares.

Note also that \eqref{nondeg2-3} is equivalent to
\begin{equation}\label{nondeg35-3}
\max\left\{ \|v_m\|_{L^\infty(B_1)}\,,\, \left|\frac{\int_{B_2} v_m(x)\,(x\cdot \nu_m)_+^s\,dx}{\int_{B_2}(x\cdot \nu_m)_+^{2s}\,dx}-\frac{\int_{B_1} v_m(x)\,(x\cdot \nu_m)_+^s\,dx}{\int_{B_1}(x\cdot \nu_m)_+^{2s}\,dx}\right|\right\}\ge 1/2,
\end{equation}
which holds for all $m\geq1$.

In addition, by definition of $\theta$, for all $k$ and $z$ we have
\[ \frac{(r')^{s-\beta}| Q_{k,z}(2r')-Q_{k,z}(r')|}{\theta(r)}\le1 \quad\mbox{for all }r'\ge r>0.\]
Thus, for $R=2^N$ we have
\[\begin{split}
\frac{r^{s-\beta}|Q_{k,z}(rR)- Q_{k,z}(r)|}{\theta(r)}&\le \sum_{j=0}^{N-1} 2^{j(\beta-s)}\frac{ (2^jr)^{s-\beta} |Q_{k,z}(2^{j+1}r)- Q_{k,z}(2^jr)|}{\theta(r)}\\
&\le \sum_{j=0}^{N-1} 2^{j(\beta-s)} \le C 2^{N(\beta-s)} = CR^{\beta-s},
\end{split}\]
where we have used $\beta>s$.

Moreover, we have
\[\begin{split}
\|v_{m}\|_{L^\infty(B_R)} &= \frac{1}{\theta({r_m})(r_m)^{\beta} } \bigl\|u_{k_m}-Q_{k_m,z_m}(r_m) \bigl((x-z_m)\cdot \nu_m\bigr)_+^{s_m}\bigr\|_{L^\infty\left(B_{r_m R}\right)}
\\
&\le \frac{R^{\beta}}{\theta({r_m}) (r_m R)^{\beta} }\bigl\|u_{k_m}- Q_{k_m,z_m}(r_mR) \bigl((x-z_m)\cdot \nu_m\bigr)_+^{s_m}\bigr\|_{L^\infty\left(B_{r_m R}\right)} \,+
\\
& \hspace{35mm} +  \frac{1}{\theta({r_m}) (r_m)^{\beta} } |Q_{k_m,z_m}(r_mR)-Q_{k_m,z_m}(r_m)|\,(r_mR)^{s_m}
\\
&\le \frac{R^{\beta}\theta({r_m}R)}{\theta({r_m})} + CR^\beta,
\end{split}\]
and hence $v_m$ satisfy the growth control
\begin{equation}\label{growthc0-3}
\|v_{m}\|_{L^\infty(B_R)} \leq CR^{\beta}\quad \textrm{for all}\ \,R\ge 1.
\end{equation}
We have used the definition $\theta(r)$ and its monotonicity.

Now, without loss of generality (taking a subsequence), we assume that
\[\nu_m\longrightarrow \nu \in S^{n-1}.\]
Then, the rest of the proof consists mainly in showing the following Claim.

\vspace{5pt}
\noindent {\bf Claim. }{\em A subsequence of $v_m$ converges locally uniformly in $\R^n$ to some function $v$ which satisfies
$Lv = 0$ in $\{x\cdot \nu >0\}$ and $v=0$ in $\{x\cdot\nu<0\}$, for some $L$ of the form \eqref{operator-L} satisfying \eqref{ellipt-const}.
}
\vspace{5pt}

Once we know this, a contradiction is immediately reached using the Liouville type Theorem \ref{Liouv-half}, as seen at the end of the proof.

To prove the Claim, given $R\ge1$ and $m$ such that $r_mR<1/2$ define
\[\Omega_{R,m}^+ = \bigl\{x\in B_R\ :\ (z_m+ r_mx)\in \Omega^+_{k_m} \quad \mbox{and}\quad x\cdot\nu_m(z_m)>0  \bigr\}.\]
Notice that for all $R$ and $k$, the origin $0$ belongs to the boundary of $\Omega_{R,m}^+$.

We will use that $v_m$ satisfies an elliptic equation in $\Omega_{R,m}^+$.
Namely,
\begin{equation}\label{eqvm0}
L_{k_m} v_m(x) =  \frac{(r_m)^{2s}}{(r_m)^\beta\theta(r_m)}f_{k_m}(z_m +r_m x) \quad \mbox{ in }\Omega_{R,m}^+.
\end{equation}
This follows from the definition of $v_m$ and the fact that $L_{k_m}\phi_m = 0$ in $\{(x-z)\cdot \nu_m>0\}$.

Notice that the right hand side of \eqref{eqvm0} converges uniformly to $0$ as $r_m \searrow 0$, since $\beta<2s$ and $\theta(r_m)\uparrow \infty$.

In order to prove the convergence of a subsequence of $v_m$, we first obtain, for every fixed $R\ge 1$, a uniform in $m$ bound for $\|v_m\|_{C^\delta (B_R)}$, for some small $\delta>0$.
Then the local uniform convergence of a subsequence of $v_m$ follows from the Arzel\`a-Ascoli theorem.

Let us fix $R\ge1$ and consider that $m$ is always large enough so that  $r_mR<1/4$.

Let $\Sigma^-_m$ be the half space which is ``tangent'' to $\Omega^-_{k_m}$ at $z_m$, namely,
\[\Sigma^-_m:= \bigl\{(x-z_m)\cdot \nu(z_m) <0 \bigr\}.\]

The first step is showing that, for all $m$ and for all $r<1/4$,
\begin{equation}\label{bound1}
\bigl\| u_{k_m}- \phi_{m} \bigr\|_{L^\infty\left(B_{r}(z_m) \cap (\Omega^-_{k_m}\cup \Sigma^-_m) \right)} \le C r^{2s}\le C r^{2s}
\end{equation}
for some  constant $C$ depending only on $s$, $\rho_0$, ellipticity constants, and dimension.

Indeed, we may rescale and slide the supersolution $\varphi_1$ from Lemma \ref{supersol} and use the fact that all points of $\Gamma_{k_m}\cap B_{3/4}$ can be touched by balls of radius $\rho_0$ contained in $\Omega^-_{k_m}$.
We obtain that
\[|u_{k_m}|\le C \bigl({\rm dist}\,(x,\Omega^-_{k_m}) \bigr)^{s},\]
with $C$ depending only on $n$, $s$, $\rho_0$, and ellipticity constants.
On the other hand, by definition of $\phi_m$ we have
\[|\phi_{m}| \le C \bigl({\rm dist}\,(x,\Sigma^-_m) \bigr)^s.\]
But by assumption, points on $\Gamma_k\cap B_{3/4}$ can be also touched by balls of radius $\rho_0$ from the $\Omega^+_{k_m}$ side, and hence we have a quadratic control (depending only on $\rho_0$) on on how $\Gamma_{k_m}$ separates from the hyperplane $\partial\Sigma^-_m$.
As a consequence, in $B_{r}(z_m) \cap (\Omega^-_{k_m}\cup \Sigma^-_m)$ we have
\[C\bigl({\rm dist}\,(x,\Omega^-_{k_m}) \bigr)^s\leq Cr^{2s}\quad\textrm{and}\quad C\bigl({\rm dist}\,(x,\Sigma^-_{m}) \bigr)^s\leq Cr^{2s}.\]
Hence, \eqref{bound1} holds.

We use now Lemma \ref{propCgammaMs} to obtain that, for some small  $\gamma\in(0,s)$,
\[\|u_{k_m}\|_{C^\gamma(B_{1/8}(z_m))} \le C\quad \mbox{for all }m.\]
On the other hand, clearly
\[\|\phi_{m}\|_{C^\gamma(B_{1/8}(z_m))} \le C\quad \mbox{for all }m.\]
Hence,
\begin{equation}\label{bound2}
\bigl\| u_{k_m}- \phi_{m} \bigr\|_{C^\gamma\left(B_{r}(z_m) \cap (\Omega^-_{k_m}\cup \Sigma^-_m) \right)} \le C.
\end{equation}

Next, interpolating \eqref{bound1} and \eqref{bound2} we obtain, for some positive $\delta<\gamma$ small enough  (depending on $\gamma$, $s$, and $\beta$),
\begin{equation}\label{bound2b}
\bigl\| u_{k_m}- \phi_{m} \bigr\|_{C^\delta \left(B_{r}(z_m) \cap (\Omega^-_{k_m}\cup \Sigma^-_m) \right)} \le C r^{\beta}.
\end{equation}
Therefore, scaling \eqref{bound2b} we find that
\begin{equation}\label{bound3}
\bigl\| v_m \bigr\|_{C^\delta \left( B_R\setminus\Omega_{R,m}^+ \right)} \le C \quad \mbox{for all }m\mbox{ with }r_mR<1/4.
\end{equation}

Next we observe that the boundary points on $\partial \Omega_{R,m}^+ \cap B_{3R/4}$ can be touched by balls of radius $(\rho_0/r_m)\ge \rho_0$ contained in $B_R\setminus \Omega_{R,m}^+$.
We then apply Lemma \ref{propCgammaMs} (rescaled) to $v_m$.
Indeed, we have that $v_m$ solves \eqref{eqvm0} and satisfies \eqref{bound3}.
Thus, we obtain, for some $\delta'\in(0,\delta)$,
\begin{equation}\label{estest}
\bigl\| v_m \bigr\|_{C^{\delta'} ( B_{R/2} )} \le C(R), \quad \mbox{for all }m\mbox{ with }r_mR<1/4,
\end{equation}
where we write $C(R)$ to emphasize the dependence on $R$ of the constant, which also depends on $s$, $\rho_0$, ellipticity constants, and dimension, but not on $m$.

As said above, the Arzel\`a-Ascoli theorem and the previous uniform (in $m$) $C^{\delta'}$ estimate  \eqref{estest} yield the local uniform convergence  in $\R^n$ of a subsequence of  $v_{m}$  to some function $v$.

In addition, by Lemma \ref{lem-subseq} there is a subsequence of $L_{k_m}$ which converges weakly to some operator $L$, which is of the form \eqref{operator-L} and satisfies \eqref{ellipt-const}.
Hence, it follows that $L v=0$ in all of $\R^n$, and thus the Claim is proved.

\vspace{2mm}

Finally, passing to the limit the growth control \eqref{growthc0-3} on $v_m$ we find $\|v\|_{L^\infty(B_R)}\le R^{\beta}$ for all $R\ge1$.
Hence, by Theorem \ref{Liouv-half}, it must be
\[v(x)=K \bigl(x\cdot \nu(z)\bigr)_+^s.\]
Passing \eqref{2-3} to the limit, we find
\[\int_{B_1} v(x) \bigl(x\cdot \nu(z)\bigr)_+^s \,dx =0.\]
But passing \eqref{nondeg35-3} to the limit, we reach the contradiction.
Thus, the Proposition is proved.
\end{proof}

Before giving the proof of Theorem \ref{thm-bdry-reg}, we prove the following.

\begin{lem}
Let $\Gamma$ be a $C^{1,1}$ surface of radius $\rho_0>0$ splitting $B_1$ into $U^+$ and $U^-$; see Definition \ref{defi-split}.
Let $d(x)= {\rm dist}\,(x,U^-)$.
Let $x_0\in B_{1/2}$ and $z\in \Gamma$ be such that
\[ {\rm dist}\,(x_0,\Gamma) = {\rm dist}\,(x_0,z) =: 2r.\]
Then,
\begin{equation}\label{efossas0}
\left\| \bigl((x-z)\cdot\nu(z)\bigr)_+^s - d^s(x)\right\|_{L^\infty(B_r(x_0))}\le Cr^{2s},
\end{equation}
\begin{equation}\label{efossas1}
 \left[ d^s - \bigl((x-z)\cdot\nu(z)\bigr)_+^s\right]_{C^{s-\epsilon}(B_r(x_0))}\le Cr^s,
\end{equation}
and
\begin{equation}\label{efossas2}
 \left[ d^{-s} \right]_{C^{s-\epsilon}(B_r(x_0))}\le Cr^{-2s+\epsilon}.
\end{equation}
The constant $C$ depends only on $\rho_0$.
\end{lem}

\begin{proof}
Let us denote
\[ \bar d(x) = \bigl((x-z)\cdot\nu(z)\bigr)_+.\]

First, since $\Gamma$ is $C^{1,1}$ with curvature radius bounded below by $\rho_0$, we have that
$|\bar d- d|\le Cr^2$ in $B_r(x_0)$, and thus \eqref{efossas0} follows.

To prove \eqref{efossas1} we use on the one hand that
\begin{equation}\label{usem1}
 \left\| \nabla d - \nabla \bar d\right\|_{L^\infty(B_r(x_0))}\le Cr,
 \end{equation}
which also follows from the fact that $\Gamma$ is $C^{1,1}$.
On the other hand, using the inequality $|a^{s-1}-b^{s-1}| \le  |a-b| \max\{ a^{s-2}, b^{s-2}\}$ for $a,b>0$,
we find
\begin{equation}\label{usem2}
\left\| d^{s-1}- {\bar d}^{s-1} \right\|_{L^\infty(B_r(x_0))}
\le C r^2 \max\left\{ \left\|d^{s-2}\right\|_{L^\infty(B_r(x_0))}\,,\, \left\| \bar d^{s-2} \right\|_{L^\infty(B_r(x_0))}\right\} \le C r^{s}.
\end{equation}
Thus, using \eqref{usem1} and \eqref{usem2}, we deduce
\[  \left[ d^s - \bar d^s\right]_{C^{0,1}(B_r(x_0))} =
\left\| d^{s-1}\nabla d - \bar d^{s-1} \nabla \bar d\,\right\|_{L^\infty(B_r(x_0))}\leq Cr^s.\]
Therefore, \eqref{efossas1} follows.

Finally, interpolating the inequalities
\[ \left[ d^{-s} \right]_{C^{0,1}(B_r(x_0))}=\|d^{-s-1}\nabla d\|_{L^\infty(B_r(x_0))}\leq C r^{-s-1} \quad\textrm{and}\quad \|d^{-s}\|_{L^\infty(B_r(x_0))}\leq Cr^{-s},\]
\eqref{efossas2} follows.
\end{proof}

We can finally give the

\begin{proof}[Proof of Theorem \ref{thm-bdry-reg}]
First, by Proposition \ref{eq-domain-Cs}, we have $\|u\|_{L^\infty(\Omega)}\leq C\|f\|_{L^\infty(\Omega)}$.
We may assume that
\[ \|u\|_{L^\infty(\R^n)}+ \|f\|_{L^\infty(\Omega^+)}\le 1.\]

Let us pick any point on $\partial\Omega$, and let us see that $u/d^s$ is $C^{s-\epsilon}$ around this point.
Rescaling and translating $\Omega$ if necessary, we may assume that $0\in \partial\Omega$, and that the sets $U^+=\Omega\cap B_1$ and $U^-=B_1\setminus\Omega$ satisfy the conditions in Definition \ref{defi-split} (with $\Gamma=B_1\cap \partial\Omega$).

Then, by Proposition \ref{proponbdryreg} we have that, for all $z\in \Gamma\cap \overline{B_{1/2}}$,
there is $Q=Q(z)$ such that
\begin{equation}\label{quinnom}
|Q(z)|\le C\quad \mbox{and}\quad \|u-Q\, \bigl((x-z)\cdot\nu(z)\bigr)_+^s\|_{L^\infty(B_{R}(z))}\le C R^{2s-\epsilon}
\end{equation}
for all $R>0$, where $C$ depends only on $n$, $s$, $\rho_0$, $\epsilon$, and ellipticity constants.

Now, to prove the $C^{s-\epsilon}$ estimate up to the boundary for $u/d^s$ we must combine a $C^s$ interior estimate
for $u$ with \eqref{quinnom}.

Let $x_0$ be a point in $\Omega^+\cap B_{1/4}$, and let $z\in \Gamma$ be such that
\[2r:={\rm dist}\,(x_0, \Gamma) = {\rm dist}\,(x_0,z) < \rho_0.\]
Note that $B_r(x_0)\subset B_{2r}(x_0)\subset\Omega^+$ and that $z\in \Gamma\cap B_{1/2}$ (since $0\in \Gamma$).

We claim now that there is $Q=Q(x_0)$ such that $|Q(x_0)|\leq C$,
\begin{equation}\label{lastgoal1}
\|u-Qd^s\|_{L^\infty(B_r(x_0))}\le C r^{2s-\epsilon},
\end{equation}
and
\begin{equation}\label{lastgoal2}
[u-Qd^s]_{C^{s-\epsilon}(B_r(x_0))}\le C r^s,
\end{equation}
where the constant $C$ depends only on $n$, $s$, $\epsilon$, $\rho_0$, and ellipticity constants.

Indeed, \eqref{lastgoal1} follows immediately combining \eqref{quinnom} and \eqref{efossas0}.

To prove \eqref{lastgoal2}, let
\[ v_r(x)= r^{-s} u(z+rx)- Q\,(x\cdot\nu(z))_+^s.\]
Then, \eqref{quinnom}  implies
\[ \|v_r\|_{L^\infty(B_4)}\le C r^{s-\epsilon}\]
and
\[ \|v_r\|_{L^\infty(B_R)}\le C r^{s-\epsilon}R^s.\]
Moreover, $v_r$ solves the equation
\[ L v_r = r^s f(z+rx)\quad \mbox{in }B_2(\tilde x_0),\]
where $\tilde x_0= (x_0-z)/r$ satisfies $|\tilde x_0-z| =2$.
Hence, using the interior estimate in Corollary \ref{cor-interior-ball-growth} we obtain $[v_r]_{C^{s-\epsilon}(B_1(\tilde x_0))}\le Cr^{s-\epsilon}$.
This yields that
\[ r^{s-\epsilon }\left[u-Q\, \bigl((x-z)\cdot\nu(z)\bigr)_+^s\right]_{C^{s-\epsilon}(B_r(x_0))} = r^s[v]_{C^{s-\epsilon}(B_1(\tilde x_0))} \le C r^s r^{s-\epsilon}.\]
Therefore, using \eqref{efossas1}, \eqref{lastgoal2} follows.

Let us finally show that \eqref{lastgoal1}-\eqref{lastgoal2} yield the desired result.
Indeed, note that, for all $x_1$ and $x_2$ in $B_{r}(x_0)$,
\[\frac{u}{d^s}(x_1)-\frac{u}{d^s}(x_2) = \frac{\bigl(u-Qd^s\bigr)(x_1)-\bigl(u-Qd^s\bigr)(x_2)}{d^s(x_1)} + \bigl(u-Qd^s\bigr)(x_2)\bigl(d^{-s}(x_1)-d^{-s}(x_2)\bigr).\]
By \eqref{lastgoal2}, and using that $d$ is comparable to $r$ in $B_r(x_0)$, we have
\[ \frac{\bigl|\bigl(u-Qd^s\bigr)(x_1)-\bigl(u-Qd^s\bigr)(x_2)\bigr|}{d^s(x_1)} \le C|x_1-x_2|^{s-\epsilon}.\]
Also, by \eqref{lastgoal1} and \eqref{efossas2},
\[  \bigl|u-Qd^s\bigr|(x_2)\bigl|d^{-s}(x_1)-d^{-s}(x_2)\bigr|\le  C|x_1-x_2|^{s-\epsilon}.\]
Therefore,
\[ [u/d^s]_{C^{s-\epsilon}(B_r(x_0))}\le C.\]
From this, we obtain the desired estimate for $\|u/d^s\|_{C^{s-\epsilon}(\Omega^+\cap B_{1/2})}$ by summing a geometric series, as in the proof of Proposition 1.1 in \cite{RS-Dir}.
\end{proof}

\section{Final comments and remarks}
\label{sec7}

Even for the fractional Laplacian, all the interior regularity results are sharp; see for example Section 7 in \cite{Bass}.
The only difference between Theorem \ref{thm-interior-ball}(b) and the classical interior estimate for the fractional Laplacian is that we need to assume that $u\in C^\alpha(\R^n)$ in order to have a $C^{\alpha+2s}$ estimate in $B_{1/2}$.
We show here that this assumption is in fact necessary.

\begin{prop}\label{contraexemple-interior}
Let $s\in (0,1)$, and let $L$ be the operator in $\R^2$ given by \eqref{op-deltas}.
Let $\alpha\in(0,s]$, and $\epsilon>0$ small.

Then, there exists a function $u$ satisfying:
\begin{itemize}
\item[(i)] $Lu=0$ in $B_1$
\item[(ii)] $u\in C^{\alpha-\epsilon}(\R^n)$
\item[(iii)] $u\equiv0$ in $B_2\setminus B_1$
\item[(iv)] $u\notin C^{\alpha+2s}(B_{1/2})$
\end{itemize}
\end{prop}

This means that in Theorem \eqref{thm-interior-ball}(b) the $C^\alpha(\R^n)$ norm on the right hand side can not be removed.

Concerning our boundary regularity result, we also expect it to be sharp for general stable operators \eqref{operator-L}-\eqref{ellipt-const}.
Indeed, while for the fractional Laplacian (and for any operator \eqref{operator-L2} with $a\in C^\infty(S^{n-1})$) one has that $(-\Delta)^s(d^s)$ is $C^\infty(\overline\Omega)$ whenever $\Omega$ is $C^\infty$ (see \cite{Grubb}), in this case we have the following.

\begin{prop}\label{prop-appendix}
There exists an operator of the form \eqref{operator-L}-\eqref{ellipt-const} and a $C^\infty$ bounded domain $\Omega\subset\R^n$ for which
\[L(d^s)\notin L^\infty(\Omega),\]
where $d(x)$ a $C^\infty$ function satisfying $d\equiv0$ in $\R^n\setminus\Omega$, and that coincides with ${\rm dist}(x,\R^n\setminus\Omega)$ in a neighborhood of $\partial\Omega$.
\end{prop}

As a consequence of the previous example, we do not expect the estimates in Theorem \ref{thm-bdry-reg} to hold at order $s$.
In other words, we do not expect $u/d^s$ to be $C^s(\overline\Omega)$.

We next show Propositions \ref{contraexemple-interior} and \ref{prop-appendix}.

\begin{proof}[Proof of Proposition \ref{contraexemple-interior}]
Let
\[u_0(x)=(x_1)_+^{\alpha-\epsilon}\eta(x),\]
where $\eta\in C^\infty_c(B_2(p))$, $p=(0,4)$, and $\eta\equiv1$ in $B_1(p)$.
Let $u$ be the solution to
\[\left\{ \begin{array}{rcll}
L u &=&0&\textrm{in }B_1 \\
u&=&u_0&\textrm{in }\R^n\setminus B_1.
\end{array}\right.\]
Then, $u$ clearly satisfies (i), (ii), (iii).

Let us show next that $u\notin C^{\alpha+2s}(B_{1/2})$ by contradiction.
Assume $u\in C^{\alpha+2s}(B_{1/2})$, and define $u_1=u\chi_{B_1}$, and $u_2=u-u_1$.
Notice that $u_1\in C^\alpha(\R^n)$ (by Proposition \ref{eq-domain-Cs}, since $Lu_1=-Lu_2\in L^\infty(B_1)$ and $\alpha\leq s$) and $u_1\in C^{\alpha+2s}(B_{1/2})$ (by Theorem \ref{thm-interior-ball}).
Thus, we have $Lu_1\in C^\alpha(B_{1/4})$.
Therefore, we also have
\[Lu_2\in C^\alpha(B_{1/4}).\]
since $Lu_2=-Lu_1$ in $B_1$.

Now recall that
\[Lw(a,b)=\int_{-\infty}^\infty\frac{w(a,b)-w(a,b+t)}{|t|^{1+2s}}\,dt+\int_{-\infty}^\infty\frac{w(a,b)-w(a+t,b)}{|t|^{1+2s}}\,dt.\]
Hence, taking the points $x_1=(0,0)$ and $x_2=(\delta,0)$, with $\delta>0$ small, we have
\[Lu_2(x_1)-Lu_2(x_2)=\int_{-\infty}^\infty\frac{u_2(\delta,t)-u_2(0,t)}{|t|^{1+2s}}\,dt,\]
where we have used that $u_2$ has support in $B_2(p)$.
Also, $u_2(0,t)=0$ for all $t$, and hence
\[Lu_2(x_1)-Lu_2(x_2)=\int_{-\infty}^\infty\frac{u_2(\delta,t)}{|t|^{1+2s}}\,dt>\int_3^4\frac{C\delta^{\alpha-\epsilon}}{|t|^{1+2s}}dt=c\delta^{\alpha-\epsilon}.\]
Therefore,
\[\frac{Lu_2(x_1)-Lu_2(x_2)}{|x_1-x_2|^\alpha}>c\delta^{-\epsilon},\]
and hence $Lu_2\notin C^\alpha(B_{1/4})$, a contradiction.
\end{proof}

We finally give the

\begin{proof}[Proof of Proposition \ref{prop-appendix}]
We take $\Omega$ to coincide with $\tilde\Omega=\{x\in \R^n\,:\, |x|>1\}$ in a neighborhood of $z_0=(1,0,...,0)$.
Then, in a neighborhood of $x_0$, we have $d^s(x)=(|x|-1)^s$.

We will show that $L(d^s)$ is not bounded in a neighborhood of $z_0$.
Equivalently, we will show that $Lu$ is not bounded in a neighborhood of $z_0$, where
\[u(x)=(|x|^2-1)^s\eta(x),\]
where $\eta$ is a smooth function satisfying $\eta\equiv1$ in $B_\delta(z_0)$ and $\eta\equiv0$ outside $B_{2\delta}(z_0)$, where $\delta>0$ is small enough so that $\partial\Omega$ coincides with $\partial\tilde\Omega$ in $B_{2\delta}(z_0)$.

We claim that $Lu$ is bounded if and only if $L(d^s)$ is bounded, because the quotient of these two functions is $C^\infty(\overline\Omega)$.
Indeed, let $w$ be any $C^\infty(\R^n)$ extension of $u/d^s|_\Omega$.
Then, we have
\[Lu=L(d^sw)=wL(d^s)+d^sLw-I_L(d^s,w),\]
where $I_L$ is the bilinear form associated to the operator $L$.
Now, $w$ is $C^\infty$ and $d^s$ is $C^s$, it turns out that $Lw$ and $I_L(d^s,w)$ belong to $L^\infty(\Omega)$.
Hence, using that $w$ is bounded by above and below by positive constants, we find that
\[Lu\in L^\infty(\Omega)\ \Longleftrightarrow\ L(d^s)\in L^\infty(\Omega),\]
as claimed.

Notice now that, since $u$ is bounded at infinity, then to prove the boundedness of $Lu(x)$ it is only important the values of $u$ in a neighborhood of $x$.

Let $x=(x_1,x')$, with $x'\in \R^{n-1}$.
Let us restrict the function $d^s$ to the hyperplane $\{x_1=1+r\}$, with $r>0$ very small.
We find that
\[\begin{split}
u(1+r,x')&=\bigl((1+r)^2+|x'|^2-1\bigr)^s\eta(1+r,x')=(2r+r^2+|x'|^2)^s\eta(1+r,x')\\
&=r^s\left(2+r+\left|\frac{x'}{\sqrt r}\right|^2\right)^s\eta(1+r,x').\end{split}\]
Thus, if we choose $L=L_1+L_2$, with $L_1$ being the $(n-1)$-dimensional fractional Laplacian in the $(x_2,...,x_n)$ variables, and $L_2$ the $1$-dimensional fractional Laplacian in the $x_1$-variable, we find that
\[Lu(1+r,0,...,0)=L_1v^{(r)}(0)+L_2v_2(r),\]
where $v^{(r)}(x')=r^s\left(2+r+\left|\frac{x'}{\sqrt r}\right|^2\right)^s\eta(1+r,x')$, and $v_2(r)=r_+^s+(-1-r)_+^s$.
Since $L_2[(r_+)^s]=0$, then $L_2v_2(r)$ is bounded for $r>0$.
Thus, to prove that $Lu$ is not bounded in $\Omega$ it suffices to show that $L_1v^{(r)}(0)\rightarrow\infty$ as $r\downarrow0$.

But, defining
\[\tilde v^{(r)}(y)=\left(2+r+\left|y\right|^2\right)^s\eta(1+r,ry)\]
we have that
\[L_1v^{(r)}(0)=L_1\tilde v^{(r)}(0).\]
Finally, as $r\downarrow 0$, we have that $\eta(1+r,ry)$ converges to the constant function 1 in all of $\R^n$, and hence it is immediate to see that
\[\lim_{r\rightarrow0}L_1\tilde v^{(r)}(0)=+\infty,\]
as desired.
\end{proof}

%\begin{rem}
%Note that in the example that we have constructed above we have $L(d^s)\approx\log d$ near $\partial\Omega$.
%
%The above example does not work if the spectral measure $\mu$ satisfies
%\[\sup_{\nu\in S^{n-1}} \int_{S^{n-1}} |\nu\cdot\theta|^{-s}d\mu(\theta)<\infty.\]
%In that case, we can prove that $L(d^s)$ belongs to $L^\infty(\Omega)$.
%\end{rem}

\end{document}